%% file: paper.tex
\documentclass[a4paper,10pt]{amsart} 
\usepackage{amsmath,amsfonts,amssymb,graphicx,url,amsthm,hyperref}
\usepackage[capitalize,nameinlink]{cleveref}
\usepackage[left=3cm,right=3cm,top=3cm,bottom=3cm]{geometry}
\usepackage[nocompress]{cite}
\usepackage[T1]{fontenc}
\usepackage{caption}
\usepackage{subcaption}
\usepackage{graphicx,color,nicefrac} 
\usepackage{amssymb, amsmath, amsbsy,mathtools}
\usepackage{amsmath,amsfonts}
\usepackage[export]{adjustbox}
\usepackage{booktabs}
\usepackage{leftindex}
\usepackage{enumitem}
\usepackage{scrextend}
\usepackage{algorithm}
\usepackage{algorithmic}
\usepackage{xcolor}

\usepackage{tikz,pgfplots,pgfplotstable}
\pgfplotsset{compat=newest}
\usetikzlibrary{arrows,decorations,backgrounds,positioning,calc,matrix}
\usepackage[mathlines]{lineno} 
\usepackage{multirow}
\usepackage{booktabs}
\newtheorem{theorem}{Theorem}[section]
\newtheorem{lemma}[theorem]{Lemma}
\newtheorem{corollary}[theorem]{Corollary}

\def\letters{a,b,c,d,e,f,g,h,i,j,k,l,m,n,o,p,q,r,s,t,u,v,w,x,y,z}
\def\Letters{A,B,C,D,E,F,G,H,I,J,K,L,M,N,O,P,Q,R,S,T,U,V,W,X,Y,Z}
\makeatletter
\@for \@l:=\Letters \do{%
  \expandafter\edef\csname\@l bb\endcsname{%
  \noexpand\ensuremath{\noexpand\mathbb{\@l}}}%
  \expandafter\edef\csname\@l bf\endcsname{{\noexpand\bf \@l}}%
  \expandafter\edef\csname\@l cal\endcsname{%
  \noexpand\ensuremath{\noexpand\mathcal{\@l}}}%
  \expandafter\edef\csname\@l eu\endcsname{%
  \noexpand\ensuremath{\noexpand\EuScript{\@l}}}%
  \expandafter\edef\csname\@l frak\endcsname{%
  \noexpand\ensuremath{\noexpand\mathfrak{\@l}}}%
  \expandafter\edef\csname\@l rm\endcsname{{\noexpand\rm \@l}}%
  \expandafter\edef\csname\@l scr\endcsname{%
  \noexpand\ensuremath{\noexpand\mathscr{\@l}}}%
}
\@for \@l:=\letters \do{%
  \expandafter\edef\csname\@l bf\endcsname{{\noexpand\bf \@l}}%
  \expandafter\edef\csname\@l frak\endcsname{%
  \noexpand\ensuremath{\noexpand\mathfrak{\@l}}}%
  \expandafter\edef\csname\@l scr\endcsname{%
  \noexpand\ensuremath{\noexpand\mathscr{\@l}}}%
}
\makeatother
\definecolor{shadecolor}{rgb}{0.6, 0.6, 0.6} 
\definecolor{darkgreen}{rgb}{0, 0.6, 0}
\newcommand{\isdef}{\mathrel{\mathrel{\mathop:}=}}
\newcommand{\defis}{\mathrel{=\mathrel{\mathop:}}}

\newcommand{\R}{\mathbb{R}}

\newcommand{\bs}{\boldsymbol}
\DeclareMathOperator{\diam}{diam}

\DeclareMathOperator\dist{dist}
\DeclareMathOperator\supp{supp}

\newcommand{\diag}{\operatorname{diag}}

\definecolor{blue}{rgb}{0,0,0}
\begin{document}
\title{Multiscale scattered data analysis in samplet coordinates}
\author{Sara Avesani}
\address{
Sara Avesani,
Istituto Eulero,
USI Lugano,
Via la Santa 1, 6962 Lugano, Svizzera.}
\email{sara.avesani@usi.ch}
\author{R\"udiger Kempf}
\address{
R\"udiger Kempf,
Applied and Numerical Analysis, 
Department of Mathematics, 
University of Bayreuth, 
95440 Bayreuth, Germany}
\email{ruediger.kempf@uni-bayreuth.de}
\author{Michael Multerer}
\address{
Michael Multerer,
Istituto Eulero,
USI Lugano,
Via la Santa 1, 6962 Lugano, Svizzera.}
\email{michael.multerer@usi.ch}
\author{Holger Wendland}
\address{
Holger Wendland,
Applied and Numerical Analysis, 
Department of Mathematics, 
University of Bayreuth, 
95440 Bayreuth, Germany}
\email{holger.wendland@uni-bayreuth.de}

\begin{abstract}
\input{paper_abstract}
\end{abstract}

\maketitle
\input{paper_text}

\bibliographystyle{plain}
\bibliography{bibliography}
\end{document}

%% file: paper_abstract.tex
We study multiscale scattered data interpolation schemes for globally supported
radial basis functions with focus on the Mat\'ern class. The multiscale
approximation is constructed through a sequence of residual corrections, where
radial basis functions with different lengthscale parameters are combined to
capture varying levels of detail. We prove that the condition numbers
of the the diagonal blocks of the corresponding multiscale system
remain bounded independently of the particular level, allowing us to use an 
iterative solver with a bounded number of iterations for the numerical solution. 
\textcolor{blue}{Employing an appropriate diagonal scaling, the multiscale 
system becomes well conditioned. We exploit this fact to derive a general error 
estimate bounding the consistency error issuing from a numerical approximation 
of the multiscale system.} 
To apply the multiscale approach to large data sets, we suggest to represent 
each level of the multiscale system in samplet coordinates. Samplets are 
localized, discrete signed measures exhibiting vanishing moments and allow for
the sparse approximation of generalized Vandermonde matrices issuing from a
vast class of radial basis functions. Given a quasi-uniform set of $N$ data
sites, and local approximation spaces with \textcolor{blue}{exponentially 
decreasing dimension}, the samplet compressed multiscale system can
be assembled with cost $\Ocal(N \log^2 N)$. The overall cost of the proposed
approach is $\Ocal(N \log^2 N)$. The theoretical findings are accompanied by
extensive numerical studies in two and three spatial dimensions.

%% file: paper_text.tex
\section{Introduction}
Interpolation using radial basis functions (RBFs) is a widely recognized
technique for fitting functions based on scattered data in $\mathbb{R}^d$, see
\cite{madych1983multivariate, micchelli1984interpolation,
narcowich1999multilevel} for some surveys. RBF interpolation allows for the
construction of approximation spaces in arbitrary dimensions and with varying
levels of smoothness. These spaces are well known to provide high-quality
approximations, see \cite{madych1990multivariate, wu1993local}. Even so, the
computation of an RBF interpolant usually entails a significant
computational cost and the corresponding linear system is typically
ill-conditioned. Even worse, as the convergence order increases, the associated
problems become successively ill-conditioned. Moreover, using globally supported
RBFs in interpolation leads to densely populated matrices, making the problem
computationally infeasible for larger datasets. To address the latter,
techniques such as interpolation using compactly supported RBFs have been
developed \cite{wendland1995piecewise,wu1995compactly}. Unfortunately, even
interpolation with these RBFs either becomes computationally too expensive,
if unscaled kernels are used, or does not even lead to convergence, if the support
is scaled proportional to the fill-distance. Motivated by balancing this trade-off
principle, the works \cite{floater1996multistep,hales2002error,
gia2010multiscale,wendland2010multiscale}, among many others, leverage the idea
of a multiscale method where approximations on different scales are computed in
a residual correction scheme. To date, the most general convergence theorems
for multiscale interpolation using positive definite, compactly supported RBFs
have been developed in \cite{wendland2010multiscale}. The advantage of the 
multiscale method is that the condition numbers of the resulting linear systems
at each level remain independent of the particular level.

\textcolor{blue}{An effective approach to address the computational challenges 
of large datasets is the greedy kernel method \cite{wenzel2023analysis}. The
greedy kernel method systematically selects a representative subset of the data 
to construct a sparse approximation. An alternative strategy involves deep kernel 
methods \cite{deep1,rebai2016deep,wilson2016deep,wenzel2024data}. These methods 
leverage deep neural networks to learn adaptive feature representations, thereby
constructing kernels that are tuned to the intrinsic structure of medium 
to high-dimensional data. Further, kernel multigrid methods 
\cite{wright2023,lawrence2024,rieger2025} have been developed for the efficient 
numerical solution of partial differential equations on complex domains such as 
manifolds or point cloud surfaces. These techniques use specialized kernels,
such as polyharmonic or homogeneous kernels, along with localized Lagrange bases.}

In this article, we study the RBF interpolation problem using globally supported
RBFs. We address the challenge of dealing with densely-populated interpolation
matrices by employing \emph{samplet compression}, a technique introduced in
\cite{harbrecht2022samplets}. Samplets are discrete signed measures constructed
such that all polynomials up to a certain degree vanish. This vanishing moment
property allows for the compression of generalized Vandermonde matrices issuing
from \emph{asymptotically smooth} RBFs, i.e., they behave like smooth functions 
apart from zero. Hence, in samplet coordinates, the resulting interpolation 
matrices become quasi-sparse, which means that they can be compressed such that
only a sparse matrix remains. Furthermore, samplets allow for a respective 
sparse matrix algebra, see \cite{HMSS24}, and provide a meaningful
interpretation of sparsity constraints for scattered data, see \cite{BHM24}. 
\textcolor{blue}{Unlike approaches that rely on the specific kernel, samplets 
can be applied to a broad class of positive definite kernels.}

However, as the samplet transform of a given generalized Vandermonde matrix
is an isometry
up to the compression error, it does not alleviate the ill-conditioning.
Therefore, to mitigate issues related to ill-conditioning, we adopt a multiscale
interpolation approach that exploits scaled versions of the same RBF, which, in
our case, remains globally supported. Especially, we extend the results from
\cite{wendland2010multiscale} for compactly supported RBFs to the class of 
Mat\'ern kernels.
\textcolor{blue}{As a consequence, the diagonally scaled multiscale system
becomes well conditioned. We exploit this fact to derive a general error 
estimate bounding the consistency error issuing from a numerical approximation
of the multiscale system. The theoretical results are corroborated by extensive
numerical studies in two and three spatial dimensions.} 

The remainder of this article is organized as follows. In Section~\ref{sec:CMM},
we revisit the classical multiscale interpolation algorithm, detailing the
relevant subspaces and the class of radial basis functions considered in this
work. In Section~\ref{sec:Cond} we carry over a known result for compactly
supported functions to
globally supported RBFs: The condition number of the generalized Vandermonde
matrix can be bounded uniformly in the number of sites, if an
appropriate scaling is applied.
\textcolor{blue}{Section~\ref{sec:AppErr} derives a general approximation error 
estimate for the case that the block matrix in the multiscale system is 
numerically approximated.}
Section~\ref{SampletsSection} provides a brief introduction of samplets, with a
focus on the key properties of the basis transformation and compression.
In Section~\ref{section:SampletAlgorithm}, we present
the multiscale interpolation algorithm in the samplet basis and 
\textcolor{blue}{provide error and cost estimates}. Finally, in 
Section~\ref{sec:Numerics}, we provide numerical tests in both two and three
spatial dimensions.
Concluding remarks are stated in Section~\ref{sec:Conclusion}.

\textcolor{blue}{
Throughout this article, to avoid the excessive use
of generic constants, we use the notation \(A\sim B\) to indicate that
there exists constants \(\underline{c},\overline{c}>0\) independent 
of \(A,B\), and any quantities they may depend on, 
such that \(\underline{c}A\leq B \leq \overline{c}A\). In addition \(C>0\) will
always denote a generic constant, whose value can change depending
on the context.}
\section{Classical multiscale method}\label{sec:CMM}
Let $ \Omega \subseteq \Rbb^d $ be a domain and \(X=\{{\bs x}_1,\ldots,
{\bs x}_N\}\subseteq\Omega\) be a set of data sites with cardinality $ N\isdef
\#X$. 
Moreover, let $f \in C(\Omega) $ be the function generating the data
values $ f_1, \dots, f_{N}$ corresponding to $ \bs x_1, \dots, \bs x_{N}$. 
Associated to the set of data sites $ X $, there are two measures, namely the
\emph{fill-distance} of $ X $ in $ \Omega $ and the \emph{separation radius} of
$ X $, i.e.,
\begin{align*}
    h_{X,\Omega} \isdef \sup_{{\bs x} \in \Omega} \min_{{\bs x} \in X}
    \| {\bs x} - {\bs x}_j \|_2,\quad
    q_X \isdef \frac{1}{2} \min_{j \neq k} \| {\bs x}_j - {\bs x}_k \|_2.
\end{align*}
We call $ X $ \emph{quasi-uniform} if there exists a constant 
$ c_{\text{qu}} > 0 $ such that
\(
    q_X \leq h_{X,\Omega} \leq c_{\text{qu}} q_X.
\)

Our goal is to interpolate or approximate the target function by means of
RBFs only relying on the given data
\(({\bs x}_1,f_1),\ldots,({\bs x}_N,f_N)\). A function 
\(\Phi\colon\Rbb^d\to\Rbb\) is said to be \emph{radial}, if and only if there
exists a univariate function \(\phi\colon[0,\infty)\to\Rbb\) such that
\(\Phi({\bs x})=\phi(\|{\bs x}\|_2)\). 
The radial function \(\Phi\) is \emph{strictly positive definite}, if and only
if there holds for any choice of distinct points
\({\bs\xi}_1,\ldots,{\bs\xi}_n\), \(n\in\Nbb\), that
\(
{\bs\alpha}^\intercal[\Phi({\bs\xi}_i-{\bs\xi}_j)]_{i,j}{\bs\alpha}>0
\ \text{for all }
{\bs\alpha}\in\Rbb^n\setminus\{{\bs 0}\}.
\)
In this case, the kernel 
\(
K({\bs x_i},{\bs x_j})
\isdef\Phi({\bs x}_i-{\bs x}_j)
\) 
is the reproducing kernel of a uniquely determined reproducing kernel
Hilbert space, its \emph{native space}, that we denote by
\(\big(\Ncal_\Phi,\langle\cdot,\cdot\rangle_{\Ncal_\Phi}\big)\).
In this context, we will use the function \(\Phi\) and its associated
kernel \(K\) synonymously.
A popular class of strictly positive RBFs are the \emph{Mat\'ern kernels} or
\emph{Sobolev splines} 
$\Phi_{\theta}\colon\Rbb^d \to \Rbb $, dependent on the hyper-parameter 
$ \theta > d/2$, that steers the smoothness of the kernel. 
These kernels are defined via the radial function
\begin{align}\label{eq:maternKernel}
\phi_{\theta}(r) = \frac{2^{1-\theta+d/2}}{\Gamma(\theta-d/2)} r^{\theta-d/2} 
K_{\theta-d/2} (r),
\quad  r \geq 0,
\end{align}
where $ \Gamma $ is the Riemann gamma function and $ K_{\nu} $ is the
modified Bessel function of the second kind of order 
\(\nu\), see \cite{MAT} for example. 
It is well-known that the Fourier transform of this radial function satisfies
\begin{align}\label{eq:algebraicDecaySobolevSpline}
\Fcal{\Phi}(\boldsymbol{\omega}) = 
(1 + \| \boldsymbol{\omega} \|_2^2)^{-\theta}, 
\quad \boldsymbol{\omega} \in \mathbb{R}^d,
\end{align}
and that, in turn, the native space of the corresponding Mat\'ern kernel
is norm-equivalent to the Sobolev space
$ H^{\theta} (\mathbb{R}^d) $, \textcolor{blue}{see
\cite[Corollary 10.13]{wendland2004scattereddata}} for example.

In this \textcolor{blue}{\emph{instationary setting}, meaning that the 
length-scale of the kernel is not coupled to the fill-distance,} it is known 
that the interpolant 
\(
s \in \Scal \isdef \operatorname{span} 
\{\Phi(\cdot - {\bs x}):{\bs x}\in X\}\subset\Ncal_\Phi
\)
converges towards the function $ f $ as $ h_{X,\Omega} \to 0 $. However,
this improved approximation comes with higher and higher numerical cost for
the iterative solution due to the ill-conditioning of the associated generalized
Vandermonde matrix. To
alleviate this cost, we can rescale the kernel, i.e., we introduce the
parameter $ \delta > 0 $ and define 
\begin{align}\label{eq:defRescaledKernel}
  \Phi_{\delta} \isdef \delta^{-d} \Phi\bigg(\frac{\cdot}{\delta}\bigg).
\end{align}
In this \emph{stationary setting}, we couple $ \delta $ proportional to
$ h_{X,\Omega} $, i.e., there is a \emph{coupling parameter} $ \eta > 0 $ such
that 
\begin{equation} \label{param:eta}
\delta = \eta h_{X,\Omega}.
\end{equation}
In this case, we can control the numerical
costs to compute the interpolant but cannot expect convergence. This is known
as the \emph{trade-off principle}.

The idea of the multiscale method is now to introduce a hierarchy
\begin{equation}\label{eq:MLhierarchy}
X_1, X_2, \ldots, X_{L} = X,
\end{equation} 
where we refer to the subscript
of \(X_\ell\) as level with cardinalities \(N_\ell\isdef\# X_\ell\) and fill-distances
$ h_{\ell} \isdef h_{X_{\ell}, \Omega} $. Although in many applications the sets
will be nested, for the multiscale method itself, this is not necessary.
We assume that the cardinalities of the sets \(X_\ell\) are increasing with the
level, i.e., $ N_{\ell} \leq N_{\ell + 1} $ or, in other words, 
$ h_{\ell} \geq h_{\ell + 1}$. We quantify this relation in 
\cref{thrm:convergenceResult}. In what follows, we consider the stationary setting 
at each level $ \ell $. This yields level-dependent parameters
$ \delta_{\ell} > 0 $ and level-dependently rescaled kernels $ K_{\ell} $ given by
\begin{equation}\label{eq:scaledKernel}
    K_{\ell}( \boldsymbol{x},\boldsymbol{y})\isdef 
    \Phi_{\ell}(\boldsymbol{x} - \boldsymbol{y}) 
    = \delta_{\ell}^{-d} \Phi \left(\frac{\boldsymbol{x} 
    - \boldsymbol{y}}{\delta_{\ell}}\right), \quad \bs x, \bs y \in \Rbb^d,\
    \ell=1,\ldots,L.
\end{equation}

Each set \(X_\ell\) in the multiscale hierarchy \eqref{eq:MLhierarchy},
together with the respective kernel, gives rise to a subspace
\begin{equation} \label{eq:subspace}
    \Scal_\ell\isdef\operatorname{span}
    \{\Phi_{\ell}(\cdot - {\bs x}):{\bs x}\in X_\ell\}\subset\Ncal_{\Phi_{\ell}},
\end{equation}
called the \emph{local approximation space}.

\textcolor{blue}{Now, given the 
multiscale sequence \(X_1,\ldots, X_L\), we seek an approximation 
in terms of successive residual corrections. To this end, we introduce functions 
$w_\ell$ of the form
\begin{equation}\label{eq:interpolant}
w_\ell=\sum_{i=1}^{N_\ell}c_{\ell, i}\Phi_{\ell}(\cdot-{\bs x}_i)\in\Scal_\ell.
\end{equation}
The coefficents $c_{\ell, i}$ are determined from decomposing the approximation
problem by a telescoping sum according to
$$
\begin{aligned}
\left.s_1\right|_{X_1}\defis \left.w_1\right|_{X_1} & =\left.f\right|_{X_1}, \\
\left.w_2\right|_{X_2} & =\left.\left(f-w_1\right)\right|_{X_2}, \\
& \ \, \vdots \\
w_L|_{X_L} & =\bigg(f-\sum_{\ell=1}^{L-1} w_\ell\bigg)\bigg|_{X_L} .
\end{aligned}
$$
This means that, at level \(\ell\), we compute the interpolant to the
residual at level $\ell - 1$, i.e.,
\begin{equation} \label{eq:residuals}
    f - (w_1 + \cdots + w_{\ell-1} )
\end{equation}
in the space $\Scal_{\ell}$.
From this, we obtain the multiscale decomposition
$
s_L\isdef(w_1+\cdots+w_L)\in\cup_{\ell=1}^L\Scal_\ell
$
which matches
$f$ at each of the data sites in $X$, i.e., \(s_L|_X=f|_X\).}

In practice, the multiscale approximation process begins with a small set of
points, where the approximation problem \eqref{eq:interpolant} is solved. Next,
the error function \eqref{eq:residuals} is computed, and finally, the latter is
approximated at a finer scale. The coefficients of the residual correction
at level $\ell$ can be computed by solving
\[
\bs K_{\ell,\ell} \bs{c}_{\ell} 
= \bs{f}_{\ell} - \sum_{\ell'=1}^{\ell-1} \bs K_{\ell,\ell ' } \bs{c}_{\ell '}
\]
with the \emph{generalized Vandermonde matrices}
\textcolor{blue}{
\(
\bs  K_{\ell,\ell ' } \isdef\big[\Phi_{\ell'}
\big({\bs x}_i^{(\ell)}-\bs{x}_j^{(\ell ')}\big)\big]_{i,j}
\in\Rbb^{N_\ell\times N_{\ell'}}
\) and the vectors
\(\bs{f}_{\ell} \isdef \bs{f}|_{X_\ell}\in\Rbb^{N_\ell},\
\bs{c}_{\ell} \isdef [{c}_{\ell,i}]_i\in\Rbb^{N_\ell}.
\)}
Consequently, the coefficient vectors \({\bs c}_1,\ldots,{\bs c}_L\) are obtained
by solving the block lower triangular system
\begin{equation}\label{eq:InterpolationSystem}
\Kcal{\bs c}\isdef\left[\begin{array}{cccc}\bs K_{1,1} & & & \\ \bs K_{2,1} &
    \bs K_{2,2} & & \\  \vdots & \vdots &
\ddots & \\ \bs K_{L,1}  &  \bs K_{L,2}&\cdots  & \bs K_{L,L}
\end{array}\right]\left[\begin{array}{c}\bs{c}_{1} \\ \bs{c}_{2}
\\ \vdots \\ \bs{c}_{L}\end{array}\right]=\left[\begin{array}{c}\bs{f}_{1} \\
\bs{f}_{2}  \\ \vdots \\ \bs{f}_{L}\end{array}\right]\defis{\bs f}.
\end{equation}

The resulting block matrix \(\Kcal\) is in general a densely-populated lower
triangular matrix as a result of the choice of globally supported RBFs for the
approximation at each level. 
There holds the following convergence result, see, e.g.,
\cite{wendland2010multiscale,wendland2017multiscale}.

\begin{theorem}\label{thrm:convergenceResult}
Let $ \Omega \subseteq \Rbb^d $ be a bounded domain with Lipschitz boundary.
Let $ X_1, X_2, \dots $ be a sequence of point sets in $ \Omega $ with 
fill-distances $ h_1, h_2, \dots $ satisfying 
    \begin{align}
        c \mu h_{\ell} \leq h_{\ell +1} \leq \mu h_{\ell}
    \end{align}
    for $ \ell = 1,2, \dots $ with fixed $ \mu \in (0,1) $, $ c \in (0,1]$
    and $ h_1 $ sufficiently small. Further, let $ \Phi\colon \R^d \to \R $ be a
    reproducing kernel for $ H^\theta(\R^d) $, i.e., its Fourier transform 
    satisfies \(
        \Fcal{\Phi}({\bs \omega}) \sim (1 + \| {\bs \omega} \|_2^2)^{-\theta},
        \ {\bs \omega} \in \R^d\)
    and let $ \Phi_{\ell} $ be the rescaled kernel as in
    \eqref{eq:defRescaledKernel} with scale factor $ \delta_{\ell}
    = \eta h_{\ell} $, \textcolor{blue}{$ \eta > 1 $}. Let $ f \in H^{\theta}(\Omega) $. 
    Then there exist constants $ C $ and $ C_1 > 0 $ such that 
    \begin{align*}
        \| f - s_L \|_{L_2(\Omega)} \leq C( C_1 \mu^\theta 
        + C_1 \eta^{-\theta})^L \| f \|_{H^{\theta}(\Omega)}.
    \end{align*}
\end{theorem}

\section{Condition number bound for generalized Vandermonde matrices}
\label{sec:Cond}
We now show that the condition number of the 
generalized Vandermonde matrices $ \boldsymbol{K}_{\ell,\ell} $,
\(\ell=1,\ldots, L\), of the scaled
kernel can be bounded by a constant, independently of the level $\ell$.
A similar result
in the context of the multiscale method for compactly supported RBFs has been
derived in \cite{townsend2012multiscale} using techniques of
\cite{narcowich1994conditionnumbers,schaback1995error}. We will reuse these
techniques, with appropriate modifications. Furthermore, for a lighter
notation, we will drop the subscripts $ \ell $ in this section.

It goes without saying that, such a bound can not be established for general
sets of data sites $ X $. Hence, for this section, we assume that $ X $ is
quasi-uniform and the parameters of the multilevel method are chosen as in
\cref{thrm:convergenceResult}.

We remark that the algebraic decay of the Fourier transform of the Mat\'ern
kernel \eqref{eq:algebraicDecaySobolevSpline} immediately yields for the 
scaled RBF \(\Phi_{\delta}\) that
\(
    \Fcal{\Phi_{\delta}}(\boldsymbol{\omega}) 
    = (1 + \delta^2 \| \boldsymbol{\omega} \|_2^2)^{-\theta}, 
    \ \boldsymbol{\omega} \in \Rbb^d.
\)
We will bound the smallest and largest eigenvalue of the matrix $ {\bs K} $
separately. We follow the ideas of \cite{townsend2012multiscale,
narcowich1994conditionnumbers,schaback1995error}.





\begin{theorem}
Let $ \Omega \subseteq \Rbb^d $ be a bounded domain and $ X \subseteq \Omega $
be a quasi-uniform set of data sites with fill-distance $ h_{X,\Omega} $ and
separation radius $ q_X $.
Let $ \Phi $ be the Mat\'ern kernel \textcolor{blue}{with smoothness parameter $ \theta$}. Let $ \Phi_{\delta} 
= \delta^{-d} \Phi(\cdot / \delta) $ be the rescaled kernel where 
$ \delta = \eta h_{X,\Omega}$ with overlap parameter $ \eta > 1 $. Then there
exists a constant $ C = \textcolor{blue}{C(\Phi,d)} $ such that the the smallest 
eigenvalue $ \lambda_{\min} $ of $ \boldsymbol{K} 
= [\Phi_{\delta}(\boldsymbol{x}_i - \boldsymbol{x}_j)]_{i,j}$ satisfies
\(
        \lambda_{\min}({\bs K}) \geq C ( \eta c_{\textnormal{qu}})^{d - 2 \theta} 
        \delta^{-d},
\)
    while the largest eigenvalue of $ {\bs K} $ can be bounded from above by 
\(
        \lambda_{\max}({\bs K}) \leq \widetilde{C} \delta^{-d},
\)
    with a constant $ \widetilde{C} = \textcolor{blue}{C(\Phi,d,\eta,c_{\textnormal{qu}})} $.
\end{theorem}

\begin{proof}
We set 
\[
\widetilde{\bs K} = \delta^d {\bs K} 
= \bigg[\Phi\bigg( \frac{{\bs x}_i - {\bs x}_j}{\delta} \bigg) \bigg]_{i,j}.
\]
The lower bound immediately follows from 
\cite[Proof of Proposition 3.4]{townsend2012multiscale}.
Hence, we only prove the upper bound for $ \lambda_{\max}({\bs K}) $. To this
end, we apply
Gershgorin's circle theorem, see, e.g., \cite[Theorem 5.1]{wendland2017numerical}
to $ \widetilde{\bs K}$. Since the diagonal elements of this matrix are
$ \Phi ( \boldsymbol{0}) $, independently of the row $ i $, we can, without
loss of generality, set $ i = 1 $ and $ \boldsymbol{x}_1 = \boldsymbol{0} $. 
We have
    \begin{align}\label{eq:proofBoundEigenvalue1}
        | \lambda - \Phi(\boldsymbol{0}) | &\leq \sum_{k=2}^N\bigg| 
        \Phi \bigg(\frac{ \boldsymbol{0} - \boldsymbol{x}_k}{\delta} \bigg)\bigg| 
        = \sum_{k=2}^N \bigg| 
        \Phi \bigg( \frac{\boldsymbol{x}_k}{\delta} \bigg) \bigg|.
    \end{align}
    To bound the right hand side of this estimate, we introduce the annuli 
    $ E_n $ given by
    \begin{align*}
        E_n = \{ \boldsymbol{x} \in \Rbb^d \; : \; n q_X 
        \leq \| \boldsymbol{x} \|_2 < (n+1) q_X \}.
    \end{align*}
    Each $ \boldsymbol{x}_k $ is contained in exactly one $ E_n $ and we can
    estimate
    \begin{align*}
        \# \{ \boldsymbol{x}_k \in E_n \} \leq 3^d n^{d-1},
    \end{align*}
    see, e.g., \cite[Proof of Theorem 12.3]{wendland2004scattereddata}.
    Inserting this into \eqref{eq:proofBoundEigenvalue1} yields
    \begin{equation}\label{eq:proofBoundEigenvalue2}
    \begin{aligned}
        | \lambda_{\max}(\widetilde{\bs K}) - \Phi(\boldsymbol{0}) | 
        &\leq \sum_{k=2}^N \bigg| \Phi 
        \bigg( \frac{\boldsymbol{x}_k}{\delta} \bigg) \bigg|
        \leq \sum_{n=1}^{\infty} 
        \sum_{\boldsymbol{x}_k \in E_n} \bigg| \Phi \bigg( 
        \frac{\boldsymbol{x}_k}{\delta} \bigg) \bigg|\\
        &\leq \sum_{n=1}^{\infty} \# \{ \boldsymbol{x}_k \in E_n \} 
        \sup_{\boldsymbol{x} \in E_n} \bigg| 
        \Phi \bigg( \frac{\boldsymbol{x}}{\delta} \bigg) \bigg|.
    \end{aligned}
    \end{equation}
    Next, we bound $ \sup_{\boldsymbol{x} \in E_n}
    | \Phi ({\boldsymbol{x}}/{\delta})| $. We use the specific form of 
    $ \Phi $, cp.\ \eqref{eq:maternKernel}, and the fact that the function 
    $ r \mapsto r^{\beta} K_{\beta}(r) $ is non-increasing on $ (0, \infty) $
    for any $ \beta \in \Rbb$, 
    see \cite[Corollary 5.12]{wendland2004scattereddata}. This yields 
    \begin{align*}
        &\hspace{-2em}\frac{\Gamma\left(\theta-\frac{d}{2}\right)}{2^{1 - \theta+\frac{d}{2}}} 
        \sup_{\boldsymbol{x} \in E_n} \bigg| 
        \Phi \bigg( \frac{\boldsymbol{x}}{\delta} \bigg) \bigg| 
        = \sup_{\boldsymbol{x} \in E_n} \bigg( 
        \frac{\| \boldsymbol{x} \|_2}{\delta} \bigg)^{\theta - \frac{d}{2}} 
        K_{\theta - \frac{d}{2}}
        \bigg( \frac{\| \boldsymbol{x} \|_2}{\delta} \bigg) \\
        &\leq \bigg(\frac{n q_X}{\delta} \bigg)^{\theta - \frac{d}{2}} 
        K_{\theta - \frac{d}{2}}\bigg(\frac{n q_X}{\delta} \bigg) \\
        &\leq \bigg(\frac{n q_X}{\delta} \bigg)^{\theta - \frac{d}{2}} 
        \sqrt{2 \pi} \bigg(\frac{\delta}{n q_X}\bigg)^{\frac{1}{2}} 
        \exp\bigg(- \frac{n q_X}{\delta} \bigg) 
        \exp \Bigg(\bigg( \theta - \frac{d}{2} \bigg)^2 \frac{2 \delta}{2n q_X}\Bigg),
    \end{align*}
    where we used the asymptotic behavior 
    \cite[Lemma 5.13]{wendland2004scattereddata} of the modified Bessel function 
    to arrive at the last inequality. Inserting this into 
    \eqref{eq:proofBoundEigenvalue2}, using the quasi-uniformity of $ X $ 
    and the coupling of $ \delta $ to $ h_{X, \Omega} $,
    as in \cref{thrm:convergenceResult}, yields
    \begin{align*}
        | \lambda_{\max}(\widetilde{\bs K}) - \Phi(\boldsymbol{0}) | 
        &\leq \sum_{n=1}^{\infty} \# \{ \boldsymbol{x}_k \in E_n \} 
        \sup_{\boldsymbol{x} \in E_n} \bigg| 
        \Phi \bigg( \frac{\boldsymbol{x}}{\delta} \bigg) \bigg| \\
        & \leq \sum_{n=1}^{\infty} 3^d n^{d-1} \sqrt{2 \pi} 
        \eta^{\frac{d}{2} - \theta} (\eta c_{\text{qu}})^{\frac{1}{2}} 
        n^{\theta - \frac{d}{2} - \frac{1}{2}} 
        \exp \bigg(- \frac{n}{\eta} + \frac{\left(\theta - \frac{d}{2}\right)^2 
        \eta c_{\text{qu}}}{n}\bigg) \\
        &\leq \sqrt{2 \pi} 3^d  \eta^{\frac{d}{2} - \theta} 
        (\eta c_{\text{qu}})^{\frac{1}{2}} \sum_{n=1}^{\infty} 
        n^{\theta - \frac{d}{2} - \frac{3}{2}} \exp \bigg(- \frac{n}{\eta} 
        + \frac{\left(\theta - \frac{d}{2}\right)^2 \eta c_{\text{qu}}}{n} \bigg).
    \end{align*}
    Since the series on the right hand side of the estimate converges,
    this yields the claimed bound for $ \lambda_{\max}$.
\end{proof}

The derived bounds on $ \lambda_{\min} $ and $ \lambda_{\max} $ lead to 
the boundedness of the condition number of the generalized Vandermonde matrices
\({\bs K}_{\ell,\ell}\) for \(\ell=1,\ldots,L\), uniformly in the number 
of data sites and the scaling parameter. \textcolor{blue}{In addition, 
by a suitable diagonal scaling, also the lower triangular system
in \eqref{eq:InterpolationSystem} becomes well conditioned. We have the
following}

\begin{corollary}\label{cor:diagCondition}
    For every $ \ell \in \Nbb $ the condition number
    $ \operatorname{cond}_2(\boldsymbol{K}_{\ell,\ell}) $ is bounded
    independently of $ N_{\ell} $ and $ \delta_{\ell} $.
\textcolor{blue}{Furthermore, introducing the block diagonal scaling
\(
\Dcal\isdef\diag(\delta_1^d{\bs I},\ldots,\delta_L^d{\bs I}),
\)
the diagonally scaled matrix \(\Dcal^{1/2}\Kcal\Dcal^{1/2}\), cp.\ 
\eqref{eq:InterpolationSystem} is well conditioned, i.e., the
corresponding condition number satisfies
\(\operatorname{cond}_2(\Dcal^{1/2}\Kcal\Dcal^{1/2})
\sim \operatorname{cond}_F(\Dcal^{1/2}\Kcal\Dcal^{1/2})\sim 1\).}
\end{corollary}

\section{Numerical approximation of the multiscale system}\label{sec:AppErr}

\textcolor{blue}{
In this paragraph, we discuss the approximation error issuing from
replacing the matrix \(\Kcal\) of the multiscale
system in \eqref{eq:InterpolationSystem} by an approximation \(\widehat{\Kcal}\).
}

\textcolor{blue}{
As has been observed in Corollary~\ref{cor:diagCondition}, the
multiscale system becomes well-conditioned by suitable diagonal scaling.
For our analysis, we therefore consider the preconditioned matrix
\(\Kcal_{\mathrm{pc}}\isdef\Dcal^{1/2}\Kcal\Dcal^{1/2}\) and its
numerical approximation 
\(\widehat{\Kcal}_{\mathrm{pc}}\isdef\Dcal^{1/2}\widehat{\Kcal}\Dcal^{1/2}\).
We have the following standard perturbation result, see, e.g.,
\cite{Hig96}.
\begin{lemma}\label{lem:perError} 
Let \(\widehat{\Kcal}\) be sufficiently close to \(\Kcal\) and
let \(\Kcal_{\mathrm{pc}}{\bs d}=\Dcal^{1/2}{\bs f}\) and
\(\widehat{\Kcal}_{\mathrm{pc}}\widehat{\bs d}=\Dcal^{1/2}{\bs f}\). Then there exists a constant \(C>0\) such that the
relative consistency error satisfies
\[
\frac{\|{\bs d}-\widehat{\bs d}\|_2}{\|{\bs d}\|_2}
\leq C
\frac{\|\Kcal_{\mathrm{pc}}-\widehat{\Kcal}_{\mathrm{pc}}\|_F}{\|\Kcal_{\mathrm{pc}}\|_F}.
\]
\end{lemma}
\begin{proof}The proof follows similarly to the proof of 
\cite[Theorem 7.2]{Hig96}. There holds 
\[
\Kcal_{\mathrm{pc}}(\widehat{\bs d}-{\bs d})
=-(\widehat{\Kcal}_{\mathrm{pc}}-\Kcal_{\mathrm{pc}}){\bs d}
+(\widehat{\Kcal}_{\mathrm{pc}}-\Kcal_{\mathrm{pc}})({\bs d}-\widehat{\bs d})
\]
Hence, by multiplying by \(\Kcal_{\mathrm{pc}}^{-1}\),
taking norms and exploiting the consistency of the Fro\-ben\-ius and the
Euclidean norm, we obtain
\[
\|{\bs d}-\widehat{\bs d}\|_2\leq\|\Kcal_{\mathrm{pc}}^{-1}\|_F
\|\widehat{\Kcal}_{\mathrm{pc}}-\Kcal_{\mathrm{pc}}\|_F\|{\bs d}\|_2+
 \|\Kcal_{\mathrm{pc}}^{-1}\|_F
\|\widehat{\Kcal}_{\mathrm{pc}}-\Kcal_{\mathrm{pc}}\|_F\|{\bs d}-\widehat{\bs d}\|_2.
\]
Factoring out common terms and simplifying the expression then results in
\[
\|{\bs d}-\widehat{\bs d}\|_2\leq
\frac{\|\Kcal_{\mathrm{pc}}^{-1}\|_F
\|\widehat{\Kcal}_{\mathrm{pc}}-\Kcal_{\mathrm{pc}}\|_F}{1-\|\Kcal_{\mathrm{pc}}^{-1}\|_F
\|\widehat{\Kcal}_{\mathrm{pc}}-\Kcal_{\mathrm{pc}}\|_F}\|{\bs d}\|_2.
\]
Dividing by \(\|{\bs d}\|_2\) and expanding the respective terms in the
numerator and denominator by 
\(\|\Kcal_{\mathrm{pc}}\|_F\) finally gives
\[
\frac{\|{\bs d}-\widehat{\bs d}\|_2}{\|{\bs d}\|_2}
\leq\frac{\operatorname{cond}_F(\Kcal_{\mathrm{pc}})}
{1-\operatorname{cond}_F(\Kcal_{\mathrm{pc}})
\|\Kcal_{\mathrm{pc}}-\widehat{\Kcal}_{\mathrm{pc}}\|_F/\|\Kcal_{\mathrm{pc}}\|_F}
\frac{\|\Kcal_{\mathrm{pc}}-\widehat{\Kcal}_{\mathrm{pc}}\|_F}
{\|\Kcal_{\mathrm{pc}}\|_F}.
\]
For a sufficiently small consistency error, we arrive at the claimed bound
using Corollary~\ref{cor:diagCondition}.
\end{proof}}

\textcolor{blue}{
To connect the previous error estimate to the consistency error of the 
multiscale interpolant, we require the following technical lemma.
\begin{lemma}\label{lem:normEq} Consider any function
\(v_\ell=\sum_{i=1}^{N_\ell} a_i
\Phi_\ell(\cdot-{\bs x}_i^{(\ell)})\in\Scal_\ell\). Then, there holds 
\[
\|v_\ell\|_{\Ncal_{\Phi_\ell}}\sim\delta_\ell^{-d/2}\|{\bs a}\|_{2}\quad\text{as well as}\quad
\|v_\ell|X_\ell\|_{\ell^2}\sim\delta_\ell^{-d}\|{\bs a}\|_{2},
\]
where \({\bs a}\isdef[a_1,\ldots, a_{N_\ell}]^\intercal\) is the coefficient
vector of \(v_\ell\).
\end{lemma}
\begin{proof}
By
the reproducing property and due to
\(\|{\bs K}_{\ell,\ell}\|_2\sim\|{\bs K}_{\ell,\ell}^{-1}\|_2\sim\delta^{-d}_\ell\), we obtain
\[
\|v_\ell\|_{\Ncal_{\Phi_\ell}}^2={\bs v}^\intercal{\bs K}_{\ell,\ell}{\bs a}
\sim\|{\bs K}_{\ell,\ell}\|_2\|{\bs a}\|_2^2\sim
\delta^{-d}_\ell\|{\bs a}\|_2^2.
\]
Taking square roots yields the first claim.
Similarly, we observe
\[
\|v_\ell|X_\ell\|_{\ell^2}^2=\|{\bs K}_{\ell,\ell}{\bs a}\|_2^2
={\bs a}^\intercal{\bs K}_{\ell,\ell}^2{\bs a}\sim\delta_\ell^{-2d}\|{\bs a}\|_2^2
\]
and the second claim is again obtained by taking square roots.
\end{proof}
}
\textcolor{blue}{
Combining the previous two lemmata, we obtain a bound on the consistency error of the
residual corrections.
\begin{theorem}\label{Theorem:stability}
Let \(w_\ell=\sum_{i=1}^{N_\ell}c_{\ell, i}\Phi_{\ell}(\cdot-{\bs x}_i)\in\Scal_\ell\) 
for \(\ell=1,\ldots, L\) denote the residual corrections obtained from the solution
of \(\Kcal{\bs c}={\bs f}\), see \eqref{eq:InterpolationSystem}, and let 
\(\widehat{w}_\ell=\sum_{i=1}^{N_\ell}\widehat{c}_{\ell, i}\Phi_{\ell}(\cdot-{\bs x}_i)\in\Scal_\ell\) 
for \(\ell=1,\ldots, L\) be the residual corrections obtained from 
solving the perturbed system \(\widehat{\Kcal}\widehat{\bs c}={\bs f}\). Then, there holds
\[
\frac{\sum_{\ell=1}^L\|w_\ell-\widehat{w}_\ell\|_{\Ncal_{\Phi_\ell}}^2}
{\sum_{\ell=1}^L\|w_\ell\|_{\Ncal_{\Phi_\ell}}^2}\leq C
\frac{\big\|\Kcal_{\mathrm{pc}}-\widehat{\Kcal}_{\mathrm{pc}}\big\|_F^2}{\|\Kcal_{\mathrm{pc}}\|_F^2}
\]
for a constant \(C>0\).
\end{theorem}
\begin{proof}
Since the solutions to the original systems and their preconditioned versions
are connected via \({\bs d}=\Dcal^{-1/2}{\bs c}\) and \(\widehat{\bs d}=\Dcal^{-1/2}\widehat{\bs c}\),
respectively, the first equivalence in Lemma~\ref{lem:normEq} yields
\[
\|w_\ell-\widehat{w}_\ell\|_{\Ncal_{\Phi_\ell}}^2\sim
\delta^{-d}_\ell\|{\bs c}_\ell-\widehat{\bs c}_\ell\|_2^2
=\|{\bs d}_\ell-\widehat{\bs d}_\ell\|_2^2
\]
as well as
\[
\|w_\ell\|_{\Ncal_{\Phi_\ell}}^2\sim
\delta^{-d}_\ell\|{\bs c}_\ell\|_2^2
=\|{\bs d}_\ell\|_2^2.
\]
In view of 
\(\|{\bs d}-\widehat{\bs d}\|_2^2=\sum_{\ell=1}^L\|{\bs d}_\ell-\widehat{\bs d}_\ell\|_2^2\)
and \(\|{\bs d}\|_2^2=\sum_{\ell=1}^L\|{\bs d}_\ell\|_2^2\),
we arrive at
\[
\frac{\sum_{\ell=1}^L\|w_\ell-\widehat{w}_\ell\|_{\Ncal_{\Phi_\ell}}^2}
{\sum_{\ell=1}^L\|w_\ell\|_{\Ncal_{\Phi_\ell}}^2}
\sim\frac{\|{\bs d}-\widehat{\bs d}\|_2^2}{\|{\bs d}\|_2^2}
\leq C
\frac{\big\|\Kcal_{\mathrm{pc}}-\widehat{\Kcal}_{\mathrm{pc}}\big\|_F^2}{\|\Kcal_{\mathrm{pc}}\|_F^2}
\]
by employing Lemma~\ref{lem:perError}. This completes the proof.
\end{proof}
}

\textcolor{blue}{
To bound the numerical approximation error between \(s_L\)
and \(\widehat{s}_L\isdef \widehat{w}_1 + \dots + \widehat{w}_L\), we have the following result,
which links this error to the approximation error of the residual corrections.
\begin{lemma}\label{lem:mlConvergence}
Under the conditions of Theorem~\ref{Theorem:stability}, the approximation 
\(\widehat{s}_L\isdef \widehat{w}_1 + \dots + \widehat{w}_L\) to \(s_L\)
satisfies the estimate
\[
\|s_L-\widehat{s}_L\|_{L^2(\Omega)}^2\leq CL\sum_{\ell=1}^L
\|w_\ell-\widehat{w}_\ell\|_{\Ncal_{\Phi_\ell}}^2.
\]
for some constant \(C>0\).
\end{lemma}
\begin{proof}
From the Cauchy–Schwarz inequality, it follows that  
\[
\|s_L-\widehat{s}_L\|^2_{L^2(\Omega)} = \bigg\|\sum_{\ell=1}^L 
(w_\ell - \widehat{w}_\ell)\bigg\|^2_{L^2(\Omega)}
\leq L \sum_{\ell=1}^L \|w_\ell - \widehat{w}_\ell\|^2_{L^2(\Omega)}.
\]
Next, we consider the sampling inequality from \cite[Theorem 3.5]{Madych06}, i.e.,
\[
    \|u\|_{L^p} \leq C\Big(h_{X,\Omega}^\theta \| u\|_{H^\theta(\Omega)} 
    + h_{X,\Omega}^{d/p} \|u|X\|_{\ell^2}\Big),\quad C > 0.
\]
Letting \(p=2\) and squaring the previous inequality results in
\[
    \|u\|_{L^2(\Omega)}^2 \leq 2C\Big(
    h_{X,\Omega}^{2\theta}\| u\|_{H^\theta(\Omega)}^2 
    + h_{X,\Omega}^{d} \|u|X\|_{\ell^2}^2\Big).
\]
Inserting this in the first inequality, we arrive at
\[
\|s_L-\widehat{s}_L\|^2_{L^2(\Omega)}
\leq 2C L\sum_{\ell=1}^L
    \big(h_{\ell}^{2\theta}\|w_\ell-\widehat{w}_\ell\|_{H^\theta(\Omega)}^2 
    + h_{\ell}^{d} \|w_\ell-\widehat{w}_\ell|X_\ell\|_{\ell^2}^2\big).
\]
Next, we have from \cite[Lemma 1]{wendland2010multiscale} that
\[
\| w_\ell - \widehat{w}_\ell\|^2_{H^\theta(\Omega)}\leq C
\delta_\ell^{-2\theta} \| w_\ell - \widehat{w}_\ell\|^2_{\Ncal_{\Phi_\ell}},\quad C>0,
\]
while the second equivalence in Lemma~\ref{lem:normEq} amounts to
\[
\|w_\ell-\widehat{w}_\ell|X_\ell\|_{\ell^2}^2\sim
\delta_\ell^{-2d}\|{\bs c}_\ell-\widehat{\bs c}_\ell\|_2^2\sim\delta_\ell^{-d}
\| w_\ell - \widehat{w}_\ell\|^2_{\Ncal_{\Phi_\ell}}.
\]
Inserting these two estimates in the previous one yields with
 \(\delta_\ell=\eta h_\ell\), cp.\ \eqref{param:eta}, that
\[
\|s_L-\widehat{s}_L\|^2_{L^2(\Omega)}
\leq C L\sum_{\ell=1}^L
    \big(\eta^{-2\theta}\| w_\ell - \widehat{w}_\ell\|^2_{\Ncal_{\Phi_\ell}}
    + \eta^{-d}\| w_\ell - \widehat{w}_\ell\|^2_{\Ncal_{\Phi_\ell}}\big),
\]
for some constant \(C>0\), from which the claim is immediately obtained.
\end{proof}
}

\textcolor{blue}{
Finally, we need a bound on the residual corrections by the right
hand side \(f\).
\begin{lemma}\label{lem:mlStability} Under the conditions of 
Theorem~\ref{thrm:convergenceResult}, there exists a
constant \(C>0\) such that
\[
\sum_{\ell=1}^L\|w_\ell\|_{\Ncal_{\Phi_\ell}}^2\leq C\|f\|^2_{H^\theta(\Omega)}.
\]
\end{lemma}
\begin{proof}
By construction, the residual correction $ w_{\ell} $ 
is the norm-minimal interpolant to the residual 
$ e_{\ell -1}\isdef f - \sum_{m=1}^{\ell-1}w_m$, 
for any $ 1 \leq \ell \leq L $, which implies
\begin{align*}
    \|w_{\ell} \|_{\Ncal_{\Phi_{\ell}}} \leq \| E e_{\ell-1} ||_{\Ncal_{\Phi_{\ell}}}
\end{align*}
with the continuous extension operator \(E\colon H^\theta(\Omega)\to H^\theta(\Rbb^d)\),
\(\|Ef\|_{H^\theta(\Rbb^d)}\leq C_E\|f\|_{H^\theta(\Omega)}\).
Now we can use the recursion estimate \cite[Theorem 1]{wendland2010multiscale} to obtain
\begin{align*}
    \| E e_{\ell-1} ||_{\Ncal_{\Phi_{\ell}}} \leq (C_1 \mu^{\theta} + C_1 \eta^{-\theta}) \| E e_{\ell-2} \|_{\Ncal_{\Phi_{\ell-1}}}.
\end{align*}
Recurring $ \ell -2 $ times, we arrive at the bound 
\begin{align*}
    \|w_{\ell} \|_{\Ncal_{\Phi_{\ell}}}\leq\| E e_{\ell-1} ||_{\Ncal_{\Phi_{\ell}}} &\leq (C_1 \mu^{\theta} + C_1 \eta^{-\theta})^{\ell-2} \| E e_{0} \|_{\Ncal_{\Phi_{1}}} \\
    &=(C_1 \mu^{\theta} + C_1 \eta^{-\theta})^{\ell-2} \| Ef \|_{\Ncal_{\Phi_{1}}} \\
    &\leq C_E (C_1 \mu^{\theta} + C_1 \eta^{-\theta})^{\ell-2} \delta_1^{-\theta} \|f \|_{H^{\theta}(\Omega)}.
\end{align*}
By this bound and by employing the summation formula 
for the geometric series, we finally estimate
\begin{align*}
\sum_{\ell=1}^L\|w_\ell\|_{\Ncal_{\Phi_\ell}}^2
&\leq\sum_{\ell=1}^L
C_E^2c(C_1 \mu^{\theta} + C_1 \eta^{-\theta})^{2(\ell-2)} \delta_1^{-2\theta} \|f \|_{H^{\theta}(\Omega)}^2\\
&=\frac{C_E^2
\delta_1^{-2\theta}}{(C_1 \mu^{\theta} + C_1 \eta^{-\theta})^{2}} \|f \|_{H^{\theta}(\Omega)}^2
\sum_{\ell=0}^{L-1}(C_1 \mu^{\theta} + C_1 \eta^{-\theta})^{2\ell}\\ 
&=\frac{C_E^2
\delta_1^{-2\theta}}{(C_1 \mu^{\theta} + C_1 \eta^{-\theta})^{2}}
\frac{1-(C_1 \mu^{\theta} + C_1 \eta^{-\theta})^{2L}}
{1-(C_1 \mu^{\theta} + C_1 \eta^{-\theta})^{2}}
\|f \|_{H^{\theta}(\Omega)}^2.
\end{align*}
This completes the proof.
\end{proof}
}

\textcolor{blue}{
Combining the previous lemmata with
Theorem~\ref{thrm:convergenceResult}, we arrive at the final error estimate.
\begin{theorem}\label{thm:combinedErrorEstimate}
Under the conditions of Theorems~\ref{thrm:convergenceResult} and
\ref{Theorem:stability}, there exists a
constant \(C>0\) such that
\[
\|f-\widehat{s}_L\|_{L^2(\Omega)}\leq C
\Bigg(( C_1 \mu^\theta 
        + C_1 \eta^{-\theta})^L 
        +\sqrt{L}\frac{\big\|\Kcal_{\mathrm{pc}}-\widehat{\Kcal}_{\mathrm{pc}}\big\|_F}{\|\Kcal_{\mathrm{pc}}\|_F}
        \Bigg)
        \| f \|_{H^{\theta}(\Omega)}.
        \]
\end{theorem}
\begin{proof}
By the triangle inequality, there holds
\[
\|f-\widehat{s}_L\|_{L^2(\Omega)}\leq\|f-{s}_L\|_{L^2(\Omega)}+\|s_L-\widehat{s}_L\|_{L^2(\Omega)}.
\]
The first term on the right hand side is bounded using 
Theorem~\ref{thrm:convergenceResult}.
For the second term, we obtain using Lemmata\ \ref{lem:mlConvergence} and
\ref{lem:mlStability} as well as Theorem~\ref{Theorem:stability} that
\begin{align*}
\|s_L-\widehat{s}_L\|_{L^2(\Omega)}^2&\leq
C L\sum_{\ell=1}^L
   \| w_\ell - \widehat{w}_\ell\|^2_{\Ncal_{\Phi_\ell}}
\leq C L\frac{\big\|\Kcal_{\mathrm{pc}}-\widehat{\Kcal}_{\mathrm{pc}}\big\|_F^2}{\|\Kcal_{\mathrm{pc}}\|_F^2}
   \sum_{\ell=1}^L\| w_\ell\|^2_{\Ncal_{\Phi_\ell}}\\
   &\leq
   C L\frac{\big\|\Kcal_{\mathrm{pc}}-\widehat{\Kcal}_{\mathrm{pc}}\big\|_F^2}{\|\Kcal_{\mathrm{pc}}\|_F^2}
   \|f\|_{H^\theta(\Omega)}^2.
\end{align*}
Taking square roots and inserting into the first inequality yields the claim.
\end{proof}
}
\section{Samplets for scattered data approximation}\label{SampletsSection}
Samplets are discrete signed measures, which exhibit vanishing moments. We
briefly recall their underlying concepts as introduced in
\cite{harbrecht2022samplets}. To this end, we consider the image of the
spaces \(\Scal_\ell\), \(\ell=1,\ldots,L\), \textcolor{blue}{defined in 
Equation~\eqref{eq:subspace}}, under the Riesz isomorphism
\(\Jcal\colon\Ncal_{\Phi_{\ell}}\to\Ncal'_{\Phi_{\ell}}\)
and define
\begin{equation} \label{eq:spaceS}
\Scal_\ell'\isdef\operatorname{span}\{\Jcal\Phi_{\ell}(\cdot - {\bs x}):
{\bs x}\in X_\ell\}=
\operatorname{span}\{\delta_{\bs x}: {\bs x}\in X_\ell\},
\end{equation}
by the identity \(\Jcal \Phi_{\ell}(\cdot - {\bs x})=\delta_{\bs x}\), where
\(\delta_{\bs x}\) is the point evaluation functional
at \({\bs x}\in\Omega\). We equip the spaces \(\Scal_\ell'\) with an inner
product, different from the
canonical one, by the defining property \(\langle\delta_{{\bs x}_i},
\delta_{{\bs x}_j}\rangle_{\Scal_\ell'}\isdef\delta_{ij}\)
for \({\bs x}_i,{\bs x}_j\in X_\ell\).

For the sake of a more lightweight notation, we drop the subscript \(\ell\) of
the spaces \(\Scal_\ell'\)
for the remainder of this paragraph and remark that each space \(\Scal_{\ell}'\)
will in general have a different multiresolution analysis, where
the maximum level \(J\) depends on \(\ell\).
Given a multiresolution analysis
\(
\Vcal_{0}'\subset\Vcal_{1}'\subset\cdots\subset\Vcal_{J}'=\Scal',
\)
we keep track of the increment of information between two consecutive levels $j$
and $j+1$. 
Since there holds $\Vcal_{j}'\subset \Vcal_{j+1}'$, we may
orthogonally decompose 
\(
\Vcal_{j+1}' = \Vcal_{j}'{\oplus}\Wcal_{j}'
\)
for a certain \emph{detail space} $\Wcal_{j}'\perp \Vcal_{j}'$. In analogy to wavelet
nomenclature, we call the elements of a basis of \(\Vcal_{0}'\) 
\emph{scaling distributions} and the elements of a basis of one of the spaces
$\Wcal_{j}'$ \emph{samplets}. This name is motivated by the idea that the basis
distributions in $\Wcal_{j}'$ are supported at a small subsample or samplet of
the data sites in \(X\). The collection of the bases of $\Wcal_{j}'$ for
\(j=0,\ldots,J-1\)\/ together with a basis of \(\Vcal_{0}'\) is called a
\emph{samplet basis} for \(\Scal'\). A samplet basis can be constructed such
that it exhibits \emph{vanishing moments} of order \(q+1\), i.e.,
\begin{equation}\label{eq:vanishingMoments}
\textcolor{blue}{\sigma_{j,k}(p)
 = 0\quad\text{for all}\ p\in\Pcal_q,}
\end{equation}
where \(\sigma_{j,k}\in \Wcal_{j}'\) is a samplet
and \(\Pcal_q\isdef\operatorname{span}
\{{\bs x}^{\bs\alpha}:\|\bs\alpha\|_1\leq q\}\)
denotes the space of polynomials of total degree at most \(q\).

The samplet construction for \(\Scal'\) is based on a \emph{cluster tree} for
the set of data sites \(X\), i.e., a tree $\mathcal{C}$ with root \(X\) such
that each node \(\tau\in\Ccal\) is the disjoint union of its children.
Such a cluster tree \(\Ccal\) directly induces a support based hierarchical
clustering of the subspace \(\Scal'\), spanned by the 
Dirac-$\delta$-distributions supported at the data sites in \(X\).
With a slight abuse of notation, we will refer to this cluster tree also by
\(\Ccal\) and to its nodes by \(\tau\). 
Given a cluster tree \(\Ccal\), a samplet basis for \(\Scal'\), where
\(\dim\Scal'=N\), can be constructed with cost \(\Ocal(N)\). Assuming 
furthermore, that \(\Ccal\) is a balanced binary tree with maximum level
\(J=\lceil\log_2(N)\rceil\), the samplet basis exhibits the 
following properties, see \cite{harbrecht2022samplets} but also 
\cite{HKS05,TW03}.

\begin{theorem}\label{theo:waveletProperties}
The samplet basis \(\bigcup_{j=0}^J\{\sigma_{j,k}\}_k\)
forms an orthonormal basis in $\Scal'$, satisfying
the following properties:
\begin{enumerate}
\item There holds $c^{-1} 2^j\leq\operatorname{dim}\Wcal_{j}'\leq c 2^j$
for a constant \(1<c\leq 2\).
\item 
The samplets have vanishing moments of order $q+1$, 
i.e., 
\((p,\sigma_{j,k})_\Omega = 0\) for all \(p\in\Pcal_q(\Omega)\),
 where \(\Pcal_q(\Omega)\) is the space of polynomials up 
 to degree \(q\).
\item 
The coefficient vector 
${\bs\omega}_{j,k}=\big[\omega_{j,k,i}\big]_i$ 
of the samplet $\sigma_{j,k}$ satisfies the bound 
\(
\|{\bs\omega}_{j,k}\|_{1}\leq c 2^{(J-j)/2}\) 
for a constant \(0<c<2\).
\item Given $f\in C^{q+1}(O)$, for some open set 
  \(O\supset\supp\sigma_{j,k}\), there holds 
\begin{equation*}
\begin{aligned}
|\sigma_{j,k}(f)|\le \bigg(\frac{d}{2}\bigg)^{q+1}
  	\frac{(\operatorname{diam}\supp\sigma_{j,k})^{q+1}}{(q+1)!}
	\|f\|_{C^{q+1}(O)}\|{\bs\omega}_{j,k}\|_{1}.
	\end{aligned}
\end{equation*}
\item For quasi-uniform sets \(X\),
there holds $\operatorname{diam}\supp\sigma_{j,k}\leq c2^{-j/d}$
for some constant \(c>0\).
\end{enumerate}
\end{theorem}

A direct consequence of the orthonormality of the samplet basis is that the
samplet transform
\(
[\sigma_{j,k}]_{j,k}={\bs T}[\delta_{{\bs x}_i}]_{i=1}^{N}
\)
satisfies \({\bs T}^\intercal{\bs T}={\bs T}{\bs T}^\intercal
={\bs I}\in\Rbb^{N\times N}\). If the samplet transform 
\({\bs T}\) is computed recursively adhering the hierarchy induced
by the cluster tree \(\Ccal\), its cost
is of order \(\Ocal(N)\).

Employing the Riesz isometry, the samplet basis induces a basis
for \(\Scal\subset\Ncal_{\Phi}\).
To this end, consider a samplet 
\(
\sigma_{j,k}=\sum_{i=1}^{N} \omega_{j,k,i}\delta_{{\bs x}_{i}}.
\)
The samplet can be identified with the function
\(
\psi_{j,k}\isdef
\sum_{i=1}^{N}\omega_{j,k,i}\Phi({\bs x}_i-\cdot)\in\Scal
\)
by means of the Riesz isometry. The vanishing moment property 
\eqref{eq:vanishingMoments} then translates to
\(
\langle\psi_{j,k},f\rangle_{\Ncal_\Phi}=0
\)
for any \(f\in\Ncal_\Phi\) which satisfies \(f|_{\supp(\sigma_{j,k})}\in\Pcal_q\). 

Furthermore, there holds
\begin{equation} \label{eq:K_sigma}
\big[\big\langle\psi_{j,k},\psi_{j',k'}\big\rangle_{\Ncal_\Phi}\big]_{j,j',k,k'}
= {\bs T}{\bs K}{\bs T}^\intercal\defis {\bs K}^{\Sigma}.
\end{equation}
This means that the Gramian of the embedded samplet basis in \(\Scal\) is
just the samplet transformed generalized Vandermonde matrix. For
\emph{asymptotically smooth} kernels \(K\), i.e.,
there exist $C,r>0$ such that for all $({\bs x},
{\bs y})\in(\Omega\times\Omega)\setminus\Delta$ 

\begin{equation}\label{HM_eg:kernel_estimate}
  \bigg|\frac{\partial^{|\bs\alpha|+|\bs\beta|}}
  	{\partial{\bs x}^{\bs\alpha}
  	\partial{\bs y}^{\bs\beta}} K({\bs x},{\bs y})\bigg|
  		\leq C \frac{(|\bs\alpha|+|\bs\beta|)!}
		{r^{|\bs\alpha|+|\bs\beta|}
		\|{\bs x}-{\bs y}\|_2^{|\bs\alpha|+|\bs\beta|}}
\end{equation}
uniformly in $\bs\alpha,\bs\beta\in\mathbb{N}^d$ apart 
from the diagonal $\Delta\isdef \{({\bs x},{\bs y})\in\Omega
\times\Omega:{\bs x} = {\bs y}\}$, the matrix \( {\bs K}^{\Sigma}\) becomes
quasi-sparse and can efficiently be computed.
More precisely, \textcolor{blue}{a compression based on cluster 
distances is applied to compute a sparsified version of the matrix 
\({\bs K}^{\Sigma}\).} In detail, there holds the following result,
see \cite{harbrecht2022samplets}.

\begin{theorem}\label{thm:compression}
Let \(X\) be quasi-uniform with \(\#X=N\) and
set all coefficients of the generalized Vandermonde matrix ${\bs K}^\Sigma$ from 
\eqref{eq:K_sigma} to zero which satisfy the admissibility condition

\begin{equation}\label{HM_eg:cutoff}
   \dist(\tau,\tau')\ge\rho\max\{\diam(\tau),\diam(\tau')\},\quad\rho>0,
\end{equation}
where \(\tau\) is the cluster supporting \(\sigma_{j,k}\) and \(\tau'\) is the 
cluster supporting \(\sigma_{j',k'}\), respectively. 
\textcolor{blue}{Then, there exists
a constant \(C>0\), such that the resulting compressed matrix 
${\bs K}^{\Sigma,\rho}$ satisfies}
\begin{equation}\label{eq:CompressionError}
  \frac{\big\|{\bs K}^\Sigma-{\bs K}^{\Sigma,\rho}\big\|_F}
  {\big\|{\bs K}^\Sigma\big\|_F} \leq C  m_q \bigg(\frac{r\rho}{d}\bigg)^{-2(q+1)},
\end{equation}
where \(m_q=\binom{q+d}{d}\) is the dimension of \(\Pcal_q\).
The compressed matrix has $\mathcal{O}(m_q^2N\log N)$ nonzero
coefficients.
\end{theorem}
 
\textcolor{blue}{
We remark that, given a radial function satisfying \eqref{HM_eg:kernel_estimate},
the corresponding scaled kernel \eqref{eq:scaledKernel} satisfies the
asymptotical smoothness estimate, cp.\ \cite{HMQ24},
\[
  \bigg|\frac{\partial^{|\bs\alpha|+|\bs\beta|}}
  	{\partial{\bs x}^{\bs\alpha}
  	\partial{\bs y}^{\bs\beta}} K_\ell({\bs x},{\bs y})\bigg|
  		\leq C\delta_\ell^{-d} \frac{(|\bs\alpha|+|\bs\beta|)!}
		{(r\delta_\ell)^{|\bs\alpha|+|\bs\beta|}
		\|({\bs x}-{\bs y})/\delta_\ell\|_2^{|\bs\alpha|+|\bs\beta|}}=
C\delta_\ell^{-d} \frac{(|\bs\alpha|+|\bs\beta|)!}
		{r^{|\bs\alpha|+|\bs\beta|}
		\|{\bs x}-{\bs y}\|_2^{|\bs\alpha|+|\bs\beta|}},
\]
which means that the dependency of the derivatives on \(\delta_\ell\)
just cancels out. Moreover, we remark that, since we only require derivatives
of the kernel up to order \(2q+1\) to be bounded, the asymptotical smoothness
requirement could in principle also be weakened to kernels of finite
smoothness, see, e.g., \cite{HMSS24} for details.
}
\section{Multiscale algorithm in samplet coordinates}
\label{section:SampletAlgorithm}
In this section, we transform the interpolation problem from the original basis,
referred to as the \emph{kernel basis}, 
to the samplet basis, as introduced in Section \ref{SampletsSection}. 
In what follows, we
denote the samplet transform for \(\Scal'_\ell\) by
\({\bs T}_\ell\in\Rbb^{N_\ell\times N_\ell}\) and recall that 
\({\bs T}^\intercal_\ell{\bs T}_{\ell}={\bs T}_{\ell}{\bs T}^\intercal_{\ell}
={\bs I}\in\Rbb^{N_\ell\times N_\ell}\). Defining the block transform
\(
\Tcal\isdef\diag(
{\bs T}_{1},\ldots,{\bs T}_{L}
)
\)
there holds \(\Tcal^\intercal\Tcal=\Tcal\Tcal^\intercal={\bs I}\).
Hence, the linear system \eqref{eq:InterpolationSystem} is equivalent to
\begin{equation*}
\begin{aligned}
\Tcal^\intercal\Tcal
\left[
\renewcommand{\arraycolsep}{2pt}
\begin{array}{cccc}\bs K_{1,1} & & & \\ 
\bs K_{2,1} & \bs K_{2,2} & & \\ 
\vdots & \vdots & \ddots & \\ 
\bs K_{L,1} &  \bs K_{L,2}& \cdots & \bs K_{L,L}
\end{array}\right]
\Tcal^\intercal\Tcal
\left[\begin{array}{c}\bs{c}_{1} \\ \bs{c}_{2}
\\ \vdots \\ \bs{c}_{L}\end{array}\right]
=\Tcal^\intercal\Tcal\left[\begin{array}{c}\bs{f}_{1} \\ \bs{f}_{2} 
 \\ \vdots \\ \bs{f}_{L}\end{array}\right].
\end{aligned}
\end{equation*}
\textcolor{blue}{
Letting
\(
    \bs K^{\Sigma}_{\ell, \ell '}\isdef {\bs T}_{\ell}
    \bs K_{\ell,\ell '} {\bs T}^\intercal_{\ell'}
    \in \mathbb{R}^{N_\ell \times N_\ell'},\
    \bs c^{\Sigma}_{\ell}\isdef{\bs T}_{\ell} \bs c_{\ell}
    \in \mathbb{R}^{N_\ell},\
    \bs f^{\Sigma}_{\ell}\isdef{\bs T}_{\ell} \bs f_{\ell}
    \in \mathbb{R}^{N_\ell},
\)
we can rewrite the previous system according to
\begin{equation}\label{eq:uncompressedSystem}
\left[\begin{array}{cccc}\bs K^{\Sigma}_{1,1} & & & \\ \bs K^{\Sigma}_{2,1} 
& \bs K^{\Sigma}_{2,2} & &  \\ 
 \vdots & \vdots & \ddots & \\ 
    \bs K^{\Sigma}_{L,1}   & 
    \bs K^{\Sigma}_{L,2} & \cdots& \bs K^{\Sigma}_{L,L}
\end{array}\right]\left[\begin{array}{c}\bs{c}^{\Sigma}_{1} \\ 
\bs{c}^{\Sigma}_{2} \\ \vdots \\ \bs{c}^{\Sigma}_{L}
\end{array}\right]
=
\left[\begin{array}{c}
\bs{f}^{\Sigma}_{1} \\ \bs{f}^{\Sigma}_{2} 
\\ \vdots \\ \bs{f}^{\Sigma}_{L}
\end{array}\right].
\end{equation} 
Since \(\Tcal\) is an isometry,
the condition number of this system is
identical to that of
\eqref{eq:InterpolationSystem}. In particular, the diagonal blocks
remain well-conditioned in accordance with Corollary~\ref{cor:diagCondition}.
Applying the matrix compression from Theorem \ref{thm:compression}
to each of the blocks \(\bs K^{\Sigma}_{\ell, \ell '}\)}, we
arrive at the \textcolor{blue}{compressed,} samplet transformed system

\begin{equation}\label{eq:InterpolationSystemSamplets}
\left[\begin{array}{cccc}\bs K^{\Sigma,\rho}_{1,1} & & & \\ \bs K^{\Sigma,\rho}_{2,1} 
& \bs K^{\Sigma,\rho}_{2,2} & &  \\ 
 \vdots & \vdots & \ddots & \\ 
    \bs K^{\Sigma,\rho}_{L,1}  & 
    \bs K^{\Sigma,\rho}_{L,2}  & \cdots& \bs K^{\Sigma,\rho}_{L,L}
\end{array}\right]\left[\begin{array}{c}\bs{c}^{\Sigma}_{1} \\ 
\bs{c}^{\Sigma}_{2} \\ \vdots \\ \bs{c}^{\Sigma}_{L}
\end{array}\right]
=
\left[\begin{array}{c}
\bs{f}^{\Sigma}_{1} \\ \bs{f}^{\Sigma}_{2} 
\\ \vdots \\ \bs{f}^{\Sigma}_{L}
\end{array}\right].
\end{equation}

\textcolor{blue}{A detailed algorithm for solving \eqref{eq:InterpolationSystemSamplets} using
block forward substitution and evaluating the resulting approximation at a set of points
can be found in Algorithm~\eqref{alg:SimplifiedAlgo}. We remark that replacing the kernel matrix
for the evaluation points by its samplet compressed version leads to a further consistency error,
which can, however, be bounded with similar ideas as in Lemma~\ref{lem:perError}.}

\begin{algorithm}
\caption{Multiscale interpolation in samplet coordinates}
\label{alg:SimplifiedAlgo}
\begin{tabular}{ll}
\textbf{\underline{Input}:} & Sets of points \(\{X_1, \ldots, X_L\}\), function $f$, parameters $\rho,q$.\\[0.2em]
\textbf{\underline{Output}:} & Coefficient vectors \(\widehat{\bs c}_1,\ldots,\widehat{\bs c}_L\) corresponding to 
\(\widehat{w}_1,\ldots,\widehat{w}_L\).
\end{tabular}
\begin{algorithmic}[1]
\vspace*{0.2em}
\FOR{$\ell = 1, 2, \ldots, L$}
    \STATE Compute ${\bs f}^\Sigma_\ell\gets{\bs T}_\ell{\bs f}_\ell\isdef f|X_\ell$.
    \FOR{$\ell' = 1, 2, \ldots, \ell$}
        \STATE Compute $\bs K^{\Sigma,\rho}_{\ell, \ell'}$ using \(q+1\) vanishing moments.
    \ENDFOR
    \STATE Compute $\widehat{\bs e}^{\Sigma}_{\ell}\gets 
    \bs f^{\Sigma}_\ell - \sum_{\ell'=1}^{\ell-1} \bs K^{\Sigma,\rho}_{\ell,\ell'}\widehat{\bs c}^{\Sigma}_{\ell'}$.
    \STATE Solve the local interpolation problem 
    $\bs K^{\Sigma,\rho}_{\ell, \ell}\widehat{\bs c}^{\Sigma}_\ell =\widehat{\bs e}^{\Sigma}_{\ell}$.
    \STATE Compute $\widehat{\bs c}_\ell\gets{\bs T}_\ell^\intercal\widehat{\bs c}^{\Sigma}_\ell$.
\ENDFOR
\end{algorithmic}
\end{algorithm}

\textcolor{blue}{
The overall compression error of \eqref{eq:InterpolationSystemSamplets},
can be estimated by summing up the block-wise errors in accordance with
Theorem~\ref{thm:compression}. The following result is immediate.
\begin{theorem}\label{thm:compErrorKK}
Let \(X_1,\ldots, X_L\) be quasi-uniform. Then, the compression error in
the linear system \eqref{eq:InterpolationSystemSamplets} can be bounded
according to
\[
  \frac{\big\|{\Kcal}^\Sigma-{\Kcal}^{\Sigma,\rho}\big\|_F}
  {\big\|{\Kcal}^\Sigma\big\|_F} \leq C  m_q (r\rho/d)^{-2(q+1)},
\]
where we denote the block matrix from \eqref{eq:uncompressedSystem}
by \({\Kcal}^\Sigma\) and the one from \eqref{eq:InterpolationSystemSamplets}
by \({\Kcal}^{\Sigma,\rho}\), respectively.
\end{theorem}
}

\textcolor{blue}{Observing that \(\Tcal\Dcal^{1/2}=\Dcal^{1/2}\Tcal\), combining the
previous theorem with Theorem~\ref{thm:combinedErrorEstimate} immediately results in an
error estimate for the approximation \(\widehat{s}_L\) to \(s_L\) obtained by solving
the samplet compressed system.}

\textcolor{blue}{
\begin{theorem}Let the conditions of Theorem~\ref{thm:combinedErrorEstimate} be satisfied.
Then, the numerical approximation \(\widehat{s}_L\) resulting from solving 
\eqref{eq:InterpolationSystemSamplets} satisfies the error estimate
\[
\|f-\widehat{s}_L\|_{L^2(\Omega)}\leq C
\Big(( C_1 \mu^\theta 
        + C_1 \eta^{-\theta})^L 
        +\sqrt{L}m_q (r\rho/d)^{-2(q+1)}
        \Big)
        \| f \|_{H^{\theta}(\Omega)}.
        \]
\end{theorem}
}
\textcolor{blue}{Finally, we have the subsequent result on the cost of our multiscale
algorithm, which is in line with a similar result for multiscale interpolation
using hierarchical matrices, see \cite{BW20}.
\begin{theorem}
Let \(X_1,\ldots, X_L\) be quasi-uniform and assume that the corresponding
cardinalities \(\{N_\ell\}_{\ell}\) form a geometric sequence, i.e.,
\(N_1>0\) and \(N_\ell = c_{\mathrm{geo}}^\ell N_1\) for some
constant \(c_{\mathrm{geo}}> 1\).
Then, the overall computational cost for assembling the linear system \eqref{eq:InterpolationSystemSamplets} and the number of non-vanishing matrix
entries are both bounded by \(\mathcal{O}(N_L \log^2 N_L)\), with
the hidden constant depending on the dimension \(m_q\) of \(\Pcal_q\).
\end{theorem}
\begin{proof}
Fix a row index \(\ell\). It has been shown within
\cite[Proof of Theorem 3.16]{AHK14} that the cost for the computation of
\(\mathbf{K}^{\Sigma,\rho}_{\ell,\ell'}\) amounts to
\(\mathcal{O}\left(N_{\ell} \log N_{\ell'}\right)\), which is also the number of
non-vanishing entries of that row.
Moreover, there are \(\ell\) blocks per row.
Hence, the cost for the computation
of the \(\ell\)-th row of \eqref{eq:InterpolationSystemSamplets} is given by
\[
\sum_{\ell'=1}^{\ell} \mathcal{O}\left(N_{\ell} \log N_{\ell'}\right) = \mathcal{O}\bigg(N_{\ell}\sum_{\ell'=1}^{\ell}\log N_{\ell'}\bigg)=
\mathcal{O}\left(N_{\ell} \log^2 N_{\ell}\right),
\]
since \(\ell=\Ocal(\log N_\ell)\) due to the geometric sequence property.
Next, summing up the cost of the individual rows yields the overall cost
\[
\sum_{\ell=1}^L\mathcal{O}\left(N_{\ell} \log^2 N_{\ell}\right)
=\mathcal{O}\bigg(\sum_{\ell=1}^LN_{\ell} \log^2 N_{\ell}\bigg)
=\mathcal{O}\bigg(\log^2 N_L\sum_{\ell=1}^LN_{\ell}\bigg)
=\mathcal{O}\left(N_L\log^2 N_L\right),
\]
again due to the geometric sequence property, i.e., \(\sum_{\ell=1}^LN_L=\Ocal(N_L)\).
This completes the proof.
\end{proof}
} 

To give an idea of the sparsity pattern of the samplet compressed system 
\eqref{eq:InterpolationSystemSamplets}, Figure~\ref{fig:Pattern} 
visualizes the matrix pattern $ \bs K^{\Sigma}_{\ell,\ell '}$
using three levels of points $X_1, X_2, X_3$ arranged on a regular grid for
the unit square, with cardinalities of 16\,641, 66\,049, and 263\,169,
respectively.
Furthermore, we remark that, for a geometric
sequence \(\{N_\ell\}_\ell\),
the cost for applying \(\Tcal\) and \(\Tcal^\intercal\), respectively, remains
linear in \(N_L\). 

\begin{figure}[htb]
    \centering
    \includegraphics[width=0.5\linewidth]{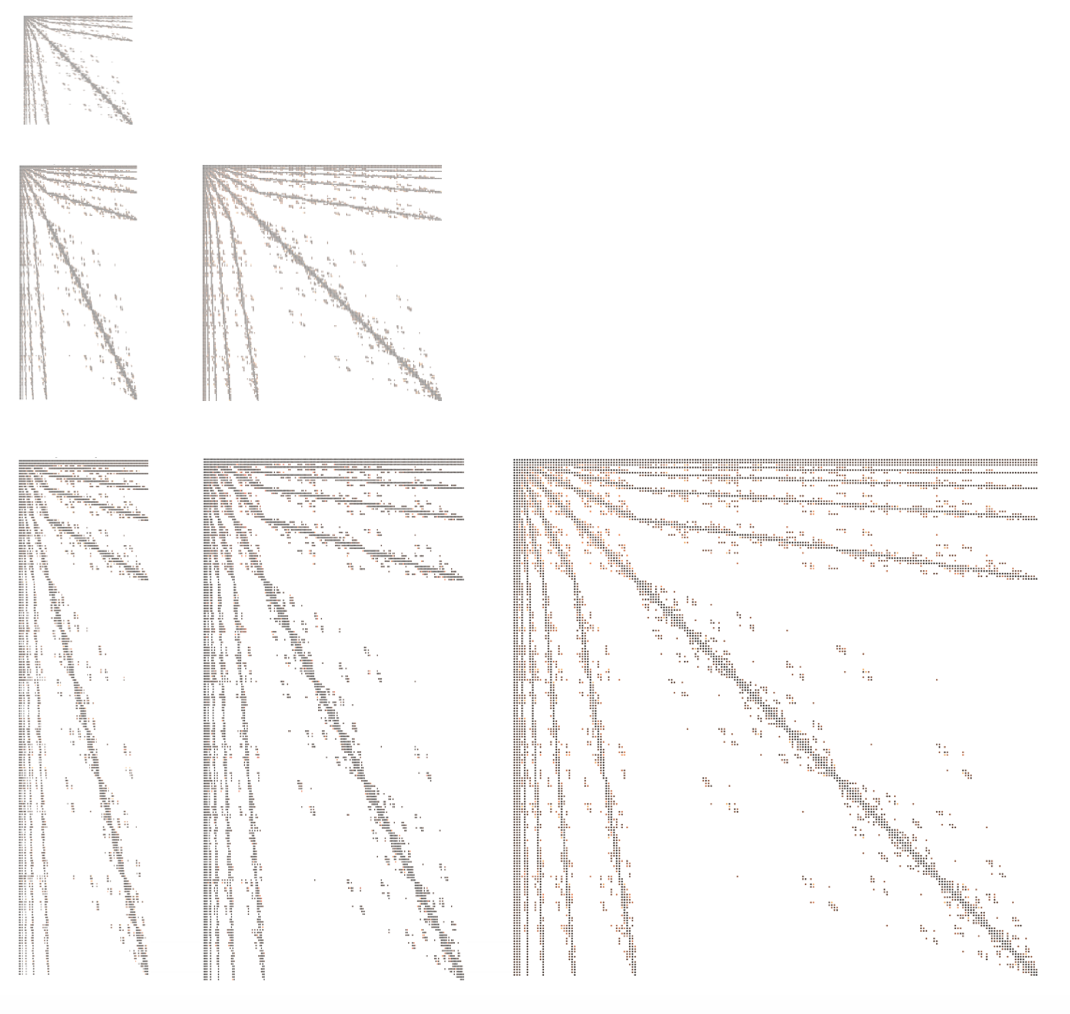}
    \caption{Example of the sparsity patterns of the matrices 
    $\bs K^{\Sigma,\rho}_{\ell, \ell '}$ for $\ell, \ell ' \in \{1,2,3\}$ 
    of the compressed linear system from \eqref{eq:InterpolationSystemSamplets}.}
    \label{fig:Pattern}
\end{figure}

\section{Numerical results}\label{sec:Numerics}
We present numerical examples of the proposed multiscale interpolation in
samplets coordinates in both two and three dimensional settings. 
The parameters for the samplet compression are $\rho$ and $q$.
\textcolor{blue}{The first one is associated with the admissibility condition
\eqref{HM_eg:cutoff}, which determines the matrix blocks that are initially truncated.}
A larger $\rho$ results in a larger distance between clusters before they
are admissible for compression. In our analysis, we adopt $\rho = d$, so that 
the bound in Equation \eqref{eq:CompressionError} is exponentially governed by 
the number of vanishing moments, denoted as $ q + 1$, see \eqref{eq:vanishingMoments}. The higher $q + 1$, the
more polynomial exactness is gained.

\textcolor{blue}{In the experiments, we vary also the parameter \(\gamma\), defined as  
\(
\gamma = \eta\mu,
\)
where \(\delta\) is the kernel rescaling factor, given by \( \delta_i = \eta h_{X_i} \) and previously defined in~\eqref{eq:defRescaledKernel}, while \(\eta\) represents the scaling factor for the fill distance between consecutive levels, i.e., \( h_{X_{\ell+1}} \approx \mu h_{X_{\ell}} \).
}

To measure the approximation quality, we consider 
the relative interpolation errors
\[
\text{error}_{p,{m}} = \frac{\| f -\widehat{s}_m| X_{\text{eval}} \|_{\ell^p}}{\| f| X_{\text{eval}}\|_{\ell^p}}
\]
for $p = 2$ and $p = \infty$ at each of the levels $m = 1, \ldots, L$ for a given set
of evaluation points \(X_{\text{eval}}\). 
Furthermore, we compute the order of convergence, given by
\[
\text{order}_{p,{m}} =\frac{
\log{\left(\displaystyle\frac{\text{error}_{p,{m+1}}}
{{\text{error}_{p,{m}}}}\right)}}
{\log{\left(\displaystyle\frac{h_{m+1}}{h_m}\right)}}.
\]

We investigate three different scenarios in our numerical tests. The first two
are adapted from \cite{wendland2010multiscale}: a smooth function on the unit
square and a non-smooth function on the L-shaped domain. Finally, we extend to 
a three dimensional setting issuing from the Lucy surface model of the 
\texttt{Stanford 3D Scanning Repository}\footnote{%
\url{https://graphics.stanford.edu/data/3Dscanrep/}}.

All computations have been performed at the Centro Svizzero di Calcolo
Scientifico (CSCS) on a single node of the Alps cluster\footnote{%
\url{https://www.cscs.ch/computers/alps}}
 with two AMD EPYC 7742
@2.25GHz CPUs with 470GB of main memory. For the computations, we have
used up to 16 cores. 
The implementation of the samplet compression is available as software package
\texttt{FMCA}\footnote{\url{ https://github.com/muchip/fmca}}.

\subsection{Franke's function in two spatial dimensions}
\label{FrankeFunctionSection}
We interpolate Franke's function, as introduced in
\cite{franke1982scattered}, defined by
\begin{equation}
\begin{aligned}
    f(x, y) = &\frac{3}{4} \exp \left( -\frac{(9x - 2)^2}{4} 
    - \frac{(9y - 2)^2}{4} \right)
              + \frac{3}{4} \exp \left( -\frac{(9x + 1)^2}{49} 
              - \frac{(9y + 1)}{10} \right) \\
              &+ \frac{1}{2} \exp \left( -\frac{(9x - 7)^2}{4} 
              - \frac{(9y - 3)^2}{4} \right) 
              - \frac{1}{5} \exp \left( -(9x - 4)^2 - (9y - 7)^2 \right)
\end{aligned}
\end{equation}
on the unit square $\Omega = [0,1]^2$. To this end, we generate 10 sets of
nested points within this domain, distributing them on a regular grid with
size $h_\ell = 2^{-\ell}$. The number of points per level and the corresponding 
fill-distances are provided in Table~\ref{tab:PointsUS}.

\begin{table}[htb]
    \centering
    \begin{tabular}{rrl}
        \toprule
        $\ell$ & \textbf{$N_\ell$} & \textbf{$h$} \\
        \midrule
        1 & 9 & 0.5 \\
        2 & 25 & 0.25 \\
        3 & 81 & 0.125 \\
        4 & 289 & 0.0625 \\
        5 & 1\,089 & 0.03125 \\
        6 & 4\,225 & 0.015625 \\
        7 & 16\,641 & 0.0078125 \\
        8 & 66\,049 & 0.00390625 \\
        9 & 263\,169 & 0.00195312 \\
        10 & 1\,050\,625 & 0.000976562 \\
        \bottomrule
    \end{tabular}
    \caption{Number of points per level and the corresponding 
fill-distances on $\Omega = [0,1]^2$ for the interpolation of Franke's
function.}
       \label{tab:PointsUS}
\end{table}

The solution is approximated with the ${C}^1$-smooth Mat\'ern-3/2 
RBF and the
${C}^2$-smooth Mat\'ern-5/2 RBF for different values of $\gamma$ and 
using $\rho = 2$, $q=4$. Furthermore, we perform an 
a-posteriori thresholding of small matrix entries with a relative threshold of \(10^{-8}\)
with respect to the Frobenius norm, cp.\ Theorem~\ref{thm:compErrorKK}.
The error is evaluated at a fine grid
with mesh-size $h = 2^{-11}$. The accuracy of the conjugate gradient method
is set to $10^{-6}$. Increasing the accuracy does not result in smaller
errors but increases the number of iterations.
\textcolor{blue}{Interpolation error, order, and compression rate (measured as the percentage of nonzero entries, \%nz), as well as (non-preconditioned) conjugate gradient iterations (CG), are presented in Tables \ref{tab:resFranke1}--\ref{tab:resFranke4}.  
}

\begin{table}[htb]
\centering \begin{tabular}{|c|c|c|c|c|c|c|c|c|} \hline 
\multirow[t]{2}{*}{ } & \multicolumn{8}{|c|} {Mat\'ern-3/2 RBF} \\ 
\hline \multirow[t]{2}{*}{ } & \multicolumn{4}{|l|}{$\gamma=0.25$} 
& \multicolumn{4}{|l|}{$\gamma=0.5$} \\ \hline $\ell$ 
& $\text{error}_2$ & $\text{order}_2$ & \%nz & CG & $\text{error}_2$ & 
$\text{order}_2$ & \%nz & CG\\ 
\hline 1 & $ 5.67 \cdot 10^{-1}$ & -- & 55 & 5 
& $ 4.84 \cdot 10^{-1}$ & -- & 55 & 6 \\ 

2 & $ 1.38 \cdot 10^{-1}$ & 2.04 & 40 & 8 
& $ 5.29 \cdot 10^{-2}$ & 2.19 & 40 & 13 \\ 

3 & $ 3.16 \cdot 10^{-2}$ & 2.12 & 43 & 9 & 
$ 9.20 \cdot 10^{-3}$ & 2.53 & 43 & 25 \\ 

4 & $ 6.76 \cdot 10^{-3}$ & 2.23 & 48 & 10 
& $ 1.15 \cdot 10^{-3}$ & 3.00 & 48 & 34\\ 

5 & $ 1.48 \cdot 10^{-3}$ & 2.19 & 22 & 10 
& $ 2.51 \cdot 10^{-4}$ & 2.53 & 28 & 39 \\ 

6 & $ 3.45 \cdot 10^{-4}$ & 2.10 & 7.8 & 10 
& $6.08 \cdot 10^{-5}$ & 2.05 & 10 & 39 \\ 

7 & $ 8.30 \cdot 10^{-5} $ & 2.06 & 2.3 & 10 
& $1.30 \cdot 10^{-5}$ & 2.22 & 3.2  & 39 \\ 

8 & $ 1.46 \cdot 10^{-5}$ & 2.51 & 0.64 & 10 
& $1.14 \cdot 10^{-6}$ & 3.51 & 0.89 & 39 \\ 

9 & $ 3.13 \cdot 10^{-6}$ & 2.22 & 0.17 & 10 
& $ 9.98 \cdot 10^{-8}$ & 3.52 & 0.23 & 39 \\ 

10 & $ 7.85 \cdot 10^{-7}$ & 2.00 & 0.04 & 10 
& $ 3.47 \cdot 10^{-8}$ & 1.52 & 0.06 & 39 \\ 
\hline \end{tabular} \caption{Results for the interpolation of Franke's
function using the ${C}^1$-smooth Mat\'ern-3/2 RBF and $\rho = 2$, $q= 4$
 as samplet compression parameters.} \label{tab:resFranke1} 
\end{table}

\begin{table}[htb]
\centering
\begin{tabular}{|c|c|c|c|c|c|c|c|c|}
\hline
\multirow[t]{2}{*}{ } & \multicolumn{8}{|c|} {Mat\'ern-3/2 RBF} \\
\hline 
\multirow[t]{2}{*}{ } & \multicolumn{4}{|l|}{$\gamma=0.75$ } 
                      & \multicolumn{4}{|l|}{$\gamma=1$ } \\
\hline 
$\ell$ & $\text{error}_2$ & $\text{order}_2$ & \%nz &  CG & 
$\text{error}_2$ & $\text{order}_2$ & \%nz &  CG\\
\hline 
1 & $ 4.84 \cdot 10^{-1}$ & -- & 55 & 6 
  & $ 4.88 \cdot 10^{-1}$ & -- & 55 & 6 \\

2 & $ 4.67 \cdot 10^{-2}$ & 3.37 & 40 & 17 
  & $ 4.51 \cdot 10^{-2}$ & 3.43 & 40 & 20 \\

3 & $ 7.97 \cdot 10^{-3}$ & 2.55 & 43 & 40 
  & $ 7.72 \cdot 10^{-3}$ & 2.55 & 43 & 48 \\

4 & $ 5.70 \cdot 10^{-4}$ & 3.81 & 48 & 63
  & $ 3.60 \cdot 10^{-4}$ & 4.43 & 48 & 96\\

5 & $ 1.09 \cdot 10^{-4}$ & 2.38 & 35 & 90 
  & $5.75 \cdot 10^{-5}$ & 3.31 & 37 & 155 \\

6 & $ 2.38 \cdot 10^{-5}$ & 2.19 & 14 & 96 
  & $1.16 \cdot 10^{-5}$ & 2.31 & 15 & 186 \\

7 & $ 4.57 \cdot 10^{-6} $ & 2.38 & 4.4 & 98 
  & $2.06 \cdot 10^{-6}$ & 2.49 & 5.0 & 195 \\

8 & $ 4.74 \cdot 10^{-7}$ & 3.26 & 1.2 & 99 
  & $ 2.16 \cdot 10^{-7}$ & 3.25 & 1.4 & 196 \\

9 & $ 3.80 \cdot 10^{-8}$ & 3.64 & 0.32 & 98 
  & $ 3.05 \cdot 10^{-8}$ & 2.83 & 0.38 & 197 \\

10 & $1.33 \cdot 10^{-8} $ & 1.51 & 0.08 & 98 
   & $ 1.41 \cdot 10^{-8}$ & 1.11 & 0.09 & 200 \\
\hline
\end{tabular}
    \caption{Results for the interpolation of Franke's
function using the ${C}^1$-smooth Mat\'ern-3/2 RBF and 
    $\rho = 2$, $q= 4$ as samplet compression parameters.}
    \label{tab:resFranke2}
\end{table}

\begin{table}[htb]
\centering
\begin{tabular}{|c|c|c|c|c|c|c|c|c|}
\hline
\multirow[t]{2}{*}{ } & \multicolumn{8}{|c|} {Mat\'ern-5/2 RBF} \\
\hline
\multirow[t]{2}{*}{ } & \multicolumn{4}{|l|}{$\gamma=0.25$ } 
                      & \multicolumn{4}{|l|}{$\gamma=0.5$ } \\
\hline 
$\ell$ & $\text{error}_2$ & $\text{order}_2$ & \%nz &  CG & 
$\text{error}_2$ & $\text{order}_2$ & \%nz &  CG\\
\hline 
1 & $ 5.45 \cdot 10^{-1}$ & -- & 55 & 5 
  & $ 4.72 \cdot 10^{-1}$ & -- & 55 & 6 \\

2 & $ 1.12 \cdot 10^{-1}$ & 2.28 & 40 & 8 
  & $ 4.51 \cdot 10^{-2}$ & 1.39 & 40 & 16 \\
 
3 & $ 2.27 \cdot 10^{-2}$ & 2.14 & 43 & 9 
  & $ 7.77 \cdot 10^{-3}$ & 2.54 & 43 & 32 \\
 
4 & $ 4.54 \cdot 10^{-3}$ & 2.32 & 48 & 10 
  & $ 6.92 \cdot 10^{-4}$ & 3.48 & 48 & 42 \\
 
5 & $ 1.00 \cdot 10^{-3}$ & 2.18 & 22 & 10 
  & $ 1.49 \cdot 10^{-4}$ & 2.21 & 25 & 48 \\

6 & $ 2.49 \cdot 10^{-4}$ & 2.01 & 7.8 & 10 
  & $3.42 \cdot 10^{-5}$ & 2.12 & 8.8 & 48 \\
 
7 & $ 6.13 \cdot 10^{-5} $ & 2.02 & 2.3 & 10 
  & $6.88 \cdot 10^{-6}$ & 2.31 & 2.7 & 48 \\

8 & $ 6.06 \cdot 10^{-6}$ & 3.34 & 0.64 & 9 
  & $8.31 \cdot 10^{-7}$ & 3.05 & 0.74 & 48 \\

9 & $ 1.15 \cdot 10^{-6}$ & 2.40 & 0.17 & 9 
  & $ 5.68 \cdot 10^{-8}$ & 3.87 & 0.19 & 48 \\

10 & $ 3.07 \cdot 10^{-7}$ & 1.91 & 0.04 & 9 
   & $ 6.86 \cdot 10^{-9}$ & 3.05 & 0.05 & 48 \\
\hline
\end{tabular}
    \caption{Results for the interpolation of Franke's
function using the ${C}^2$-smooth Mat\'ern-5/2 RBF and 
$\rho = 2$, $q=4$ as samplet compression parameters.}
    \label{tab:resFranke3}
\end{table}

\begin{table}[htb]
\centering
\begin{tabular}{|c|c|c|c|c|c|c|c|c|}
\hline
\multirow[t]{2}{*}{ } & \multicolumn{8}{|c|} {Mat\'ern-5/2 RBF} \\
\hline 
\multirow[t]{2}{*}{ } & \multicolumn{4}{|l|}{$\gamma=0.75$ } 
                      & \multicolumn{4}{|l|}{$\gamma=1$ } \\
\hline 
$\ell$ & $\text{error}_2$ & $\text{order}_2$ & \%nz &  CG 
       & $\text{error}_2$ & $\text{order}_2$ & \%nz &  CG\\
\hline 
1 & $ 4.79 \cdot 10^{-1}$ & -- & 55 & 6 
  & $ 9.16 \cdot 10^{-1}$ & -- & 55 & 7 \\

2 & $ 4.41 \cdot 10^{-2}$ & 3.44 & 40 & 20 
  & $ 4.69 \cdot 10^{-2}$ & 4.29 & 40 & 25 \\
 
3 & $ 7.13 \cdot 10^{-3}$ & 2.63 & 43 & 58 
  & $ 7.47 \cdot 10^{-3}$ & 2.65 & 43 & 83 \\
 
4 & $ 3.29 \cdot 10^{-4}$ & 4.43 & 48 & 96 
  & $ 1.63 \cdot 10^{-4}$ & 2.19 & 48 & 171\\
 
5 & $ 6.14 \cdot 10^{-5}$ & 2.42 & 35 & 139 
  & $3.96 \cdot 10^{-5}$ & 2.04 & 37 & 301 \\
 
6 & $ 1.22 \cdot 10^{-5}$ & 2.33 & 13 & 155
  & $1.34 \cdot 10^{-5}$ & 1.56 & 15 & 369 \\
 
7 & $ 2.11 \cdot 10^{-6} $ & 2.53 & 4.2 & 157 
  & $2.06 \cdot 10^{-6}$ & 2.70 & 4.8 & 403 \\
 
8 & $ 2.57 \cdot 10^{-7}$ & 3.03 & 1.2 & 158 
  & $ 2.39 \cdot 10^{-7}$ & 3.10 & 1.4 & 407 \\

9 & $ 4.06 \cdot 10^{-8}$ & 2.66 & 0.31 & 159 
  & $ 7.56 \cdot 10^{-8}$ & 1.66 & 0.36 & 436 \\

10 & $ 3.04 \cdot 10^{-8} $ & 0.42 & 0.08 & 162 
   & $ 7.01 \cdot 10^{-8}$ & 0.11 & 0.09 & 430 \\
\hline
\end{tabular}
    \caption{Results for the interpolation of Franke's
function using the ${C}^2$-smooth Mat\'ern-5/2 RBF and 
$\rho = 2$, $q=4$ as samplet compression parameters.}
    \label{tab:resFranke4}
\end{table}

The tables clearly demonstrate the expected exponential decay of the
approximation error with increasing level for both error measures and
both RBFs under consideration.
Furthermore, we see the effect of a larger value of $\gamma$, which corresponds
to a larger lengthscale parameter of the RBF, resulting in a larger condition number. For increasing $\gamma$, we require successively more
iterations for the conjugate gradient method to achieve the prescribed accuracy, with up to 430 iterations for \(\gamma=1\) and the Mat\'ern-5/2 RBF. 

\textcolor{blue}{To validate the theoretical bound on the computational cost of assembling the 
linear system stated in Theorem~\ref{thrm:convergenceResult}, Figure~\ref{fig:PlotFrankeTimes} shows the 
single threaded computation times obtained with $\gamma = 1$ and the Mat\'ern-3/2 RBF. 
The dashed lines correspond to the rates $N \log^{\alpha} N$ 
for $\alpha = 1,2,3$. The results clearly confirm that the overall time scales as $N \log^2 N$, 
in agreement with the theoretical analysis.
}
\begin{figure}[htb]
	\begin{center}
		\begin{tikzpicture}[scale=1]
			\begin{loglogaxis}[
				grid = both,
				xlabel = {Number of points $N$}, 
				ylabel = {Time (s)}, 
				y label style={rotate=90},
				legend style = {
					font =\scriptsize,
					legend cell align=left,
					mark options={scale=1},
					/tikz/every even column/.append style={column sep=1cm}
				},
				xmax = 1e7,
				x label style= {rotate=0},
				every axis x label/.append style={at={(rel axis cs: 0.5, -0.15)}, below},
				every axis y label/.append style={at={(rel axis cs: -0.15, 0.5)}, rotate=90, below}, 
				legend style={font=\footnotesize, at={(0.97,0.4)}, anchor=north east}, 
				height=6.85cm, width=9cm 
				]
				\addplot[color={rgb,255:red,0; green,87; blue,231},mark=asterisk,mark size=2pt, line width=1pt] table[x index=0, y index=1, mark=square*, col sep=space] {timings_data.txt};
				\addlegendentry{Data}
				\addplot[color={rgb,255:red,255; green,167; blue,0},dashed, line width=1.2pt]  table[x index=0, y index=2, mark=star, col sep=space] {timings_data.txt};
                \addlegendentry{$ c N \log N$}
                \addplot[color={rgb,255:red,214; green,45; blue,32},dashed, line width=1.2pt]  table[x index=0, y index=3, mark=star, col sep=space] {timings_data.txt};
                \addlegendentry{$ c N \log^2N$}
				\addplot[color={rgb,255:red,0; green,135; blue,68},dashed, line width=1.2pt]  table[x index=0, y index=4, mark=star, col sep=space] {timings_data.txt};
				\addlegendentry{$ c N \log^3N$}
			\end{loglogaxis}
		\end{tikzpicture}
		\caption{\label{fig:PlotFrankeTimes}
			Comparison of the overall assembly time (blue) with theoretical growth rates $N \log^{\alpha} N$ for $\alpha = 1,2,3$.}
	\end{center}
\end{figure}
Finally, Figure~\ref{fig:PlotFranke} shows a visualization of the solution (left) and
the residuals (right) at 
levels \(\ell=1, 4, 7,10\) for $\gamma = 0.5$ using the Mat\'ern-3/2 RBF.

\begin{figure}[htb]
  \centering
  \begin{tikzpicture}
      \draw(0,0) node {
      \includegraphics[width=0.39\textwidth]{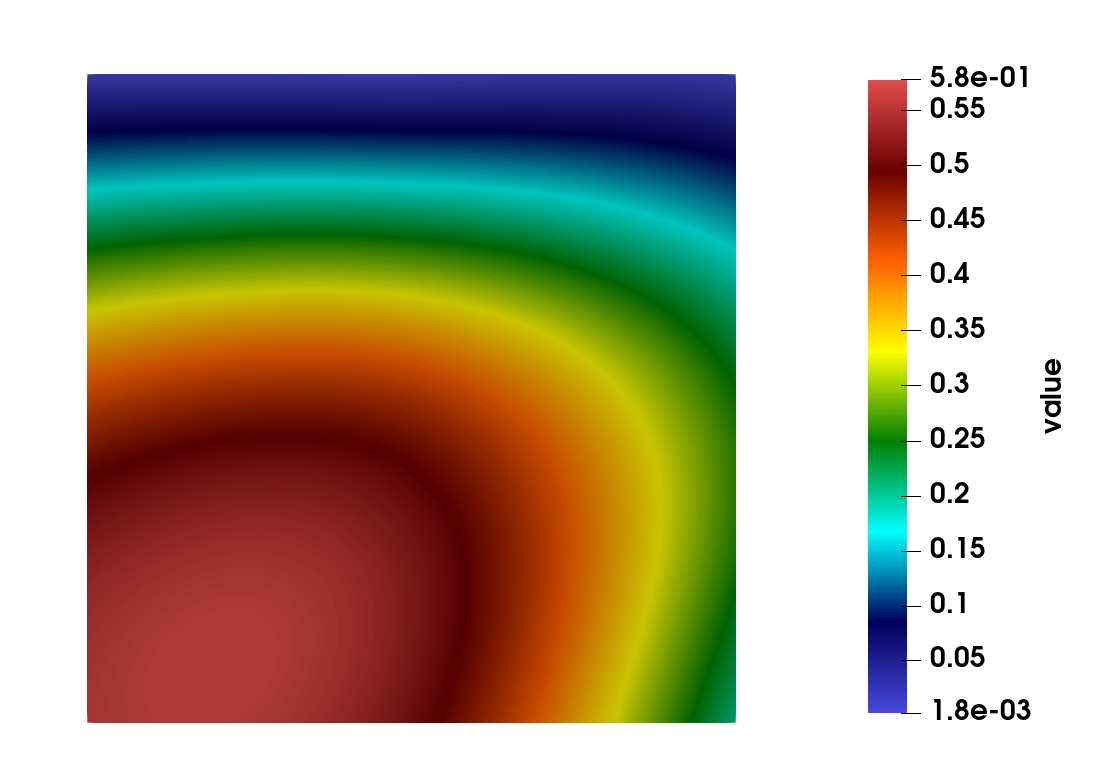}};
      \draw(6,0) node {
          \includegraphics[width=0.39\textwidth]{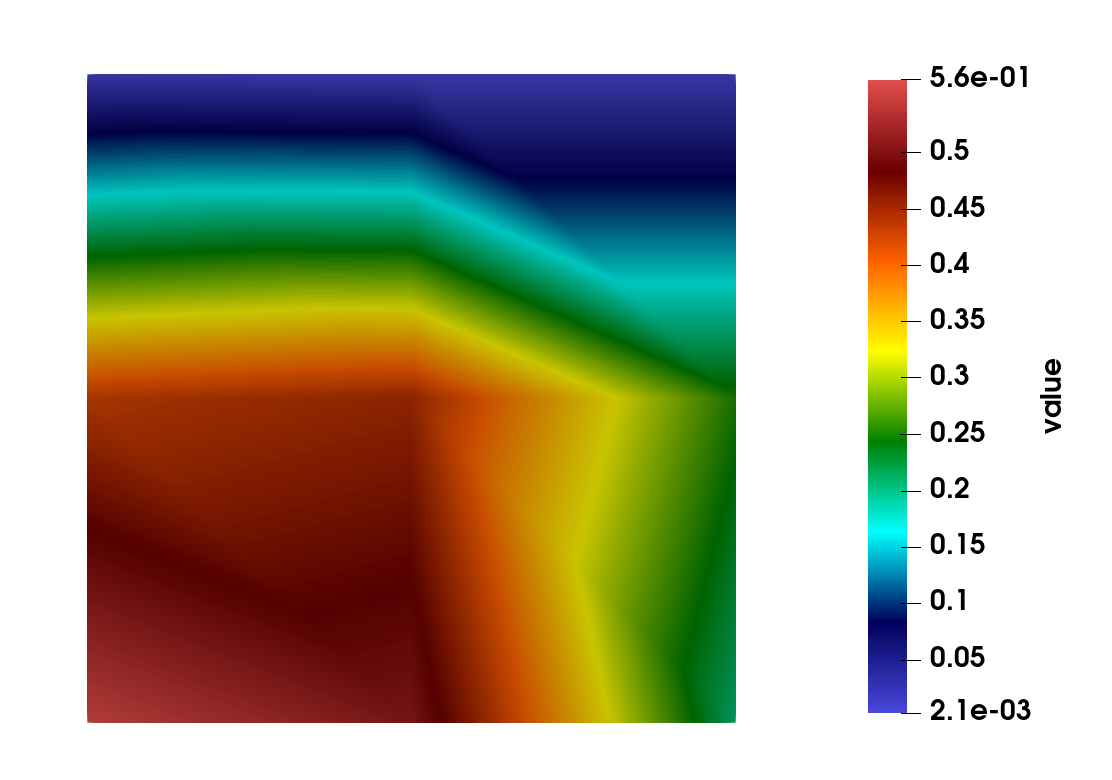}};
            \draw(3,2) node {$\ell=1$};
      \draw(0,-4.5) node {
          \includegraphics[width=0.39\textwidth]{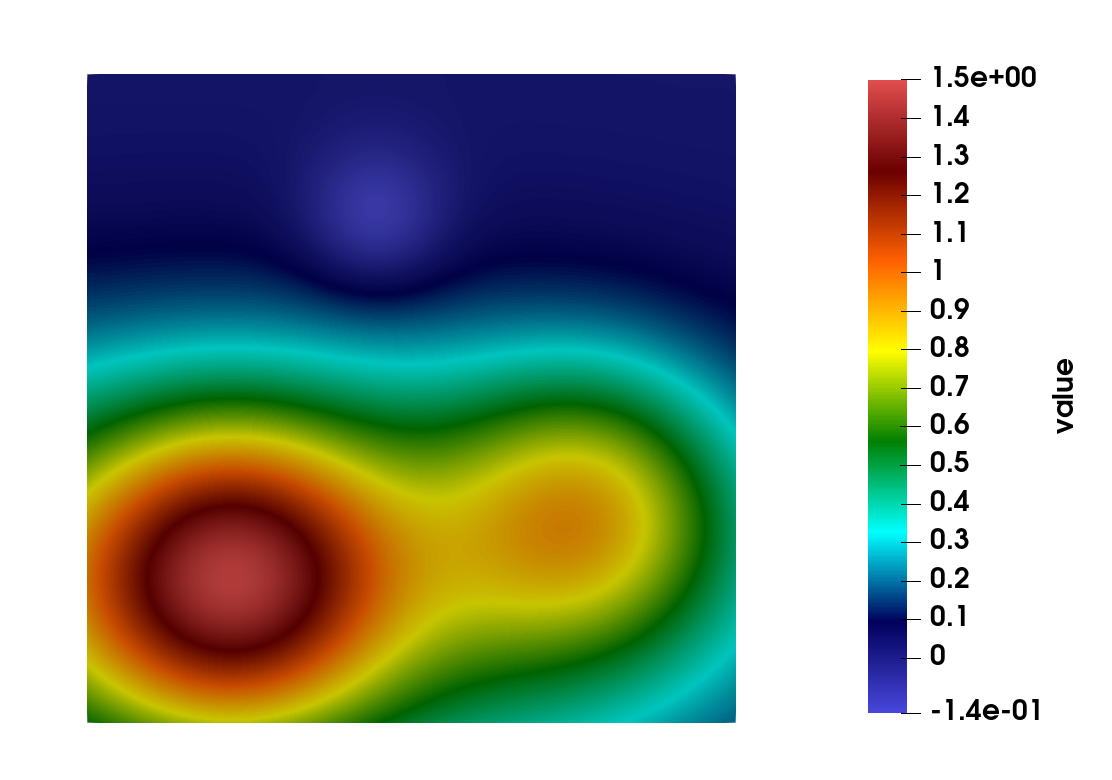}};
        \draw(6,-4.5) node {
        \includegraphics[width=0.39\textwidth]{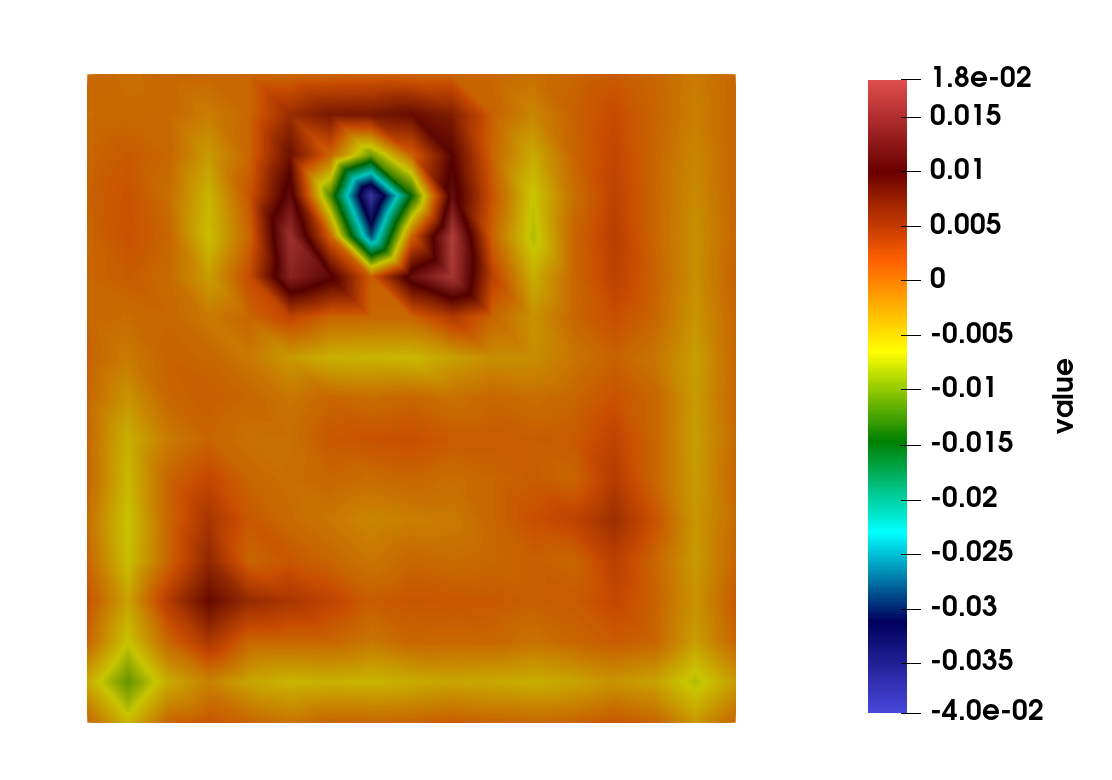}};
          \draw(3,-2.5) node {$\ell=4$};
          \draw(0,-9) node {
      \includegraphics[width=0.39\textwidth]{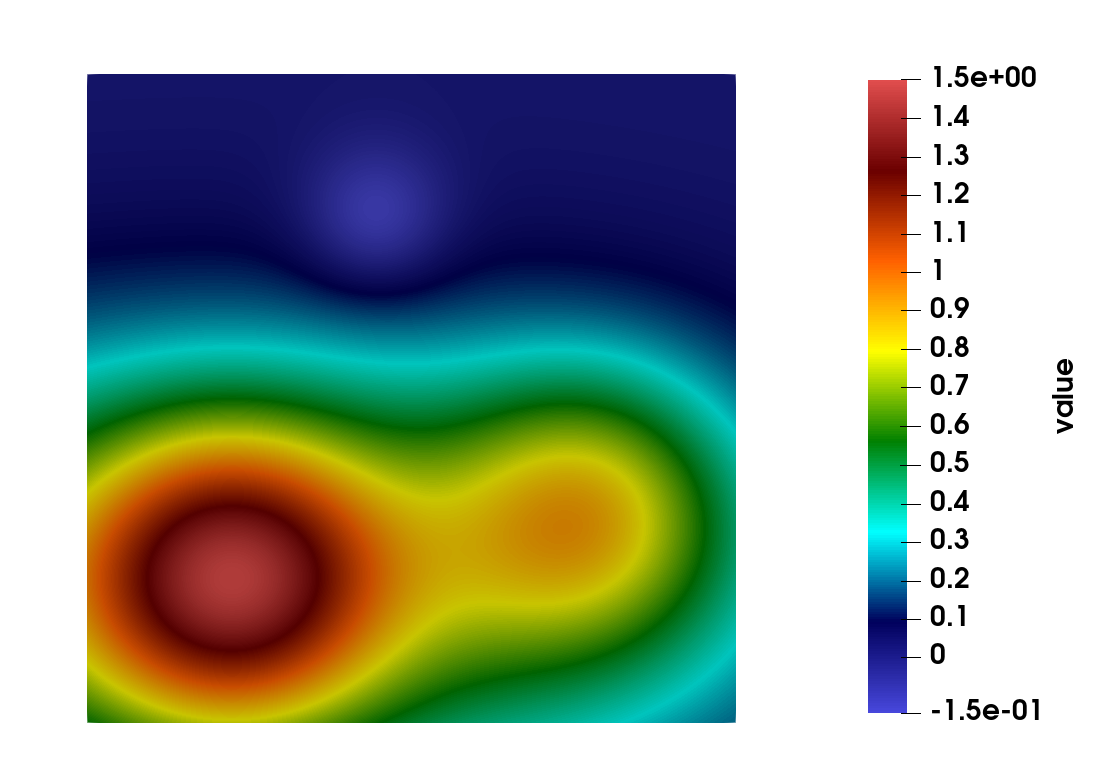}};
          \draw(6,-9) node {
      \includegraphics[width=0.39\textwidth]{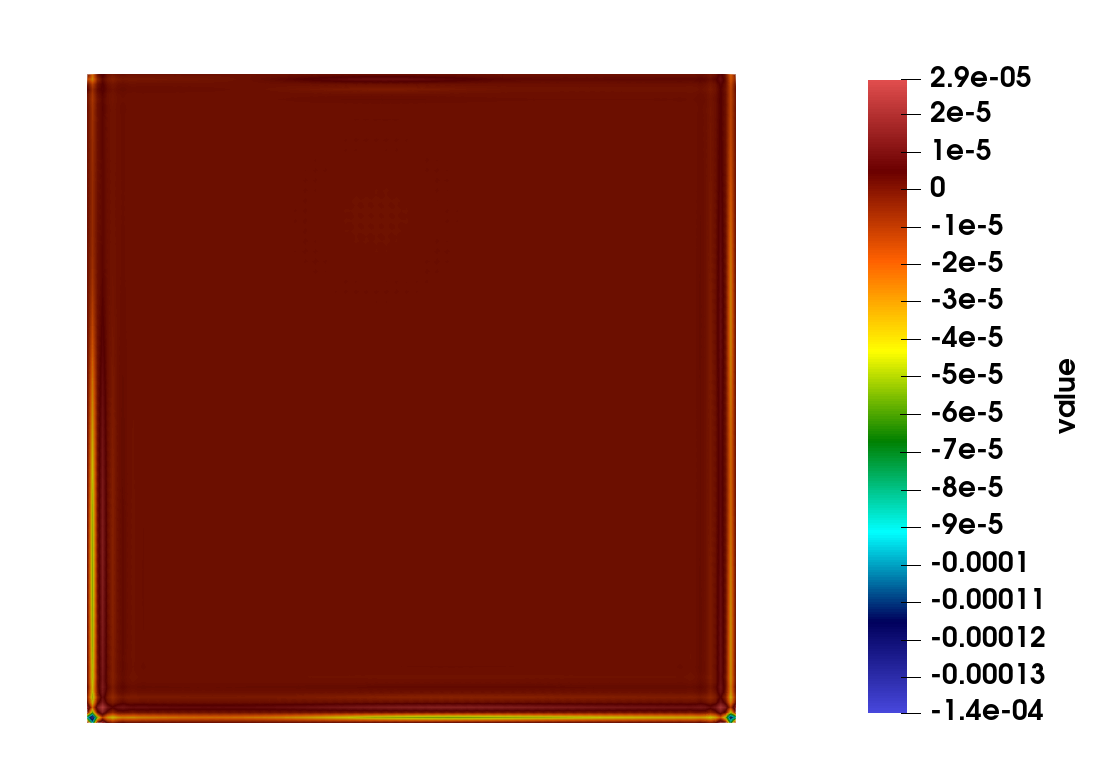}};   
    \draw(3,-7) node {$\ell=7$};
          \draw(0,-13.5) node {
      \includegraphics[width=0.39\textwidth]{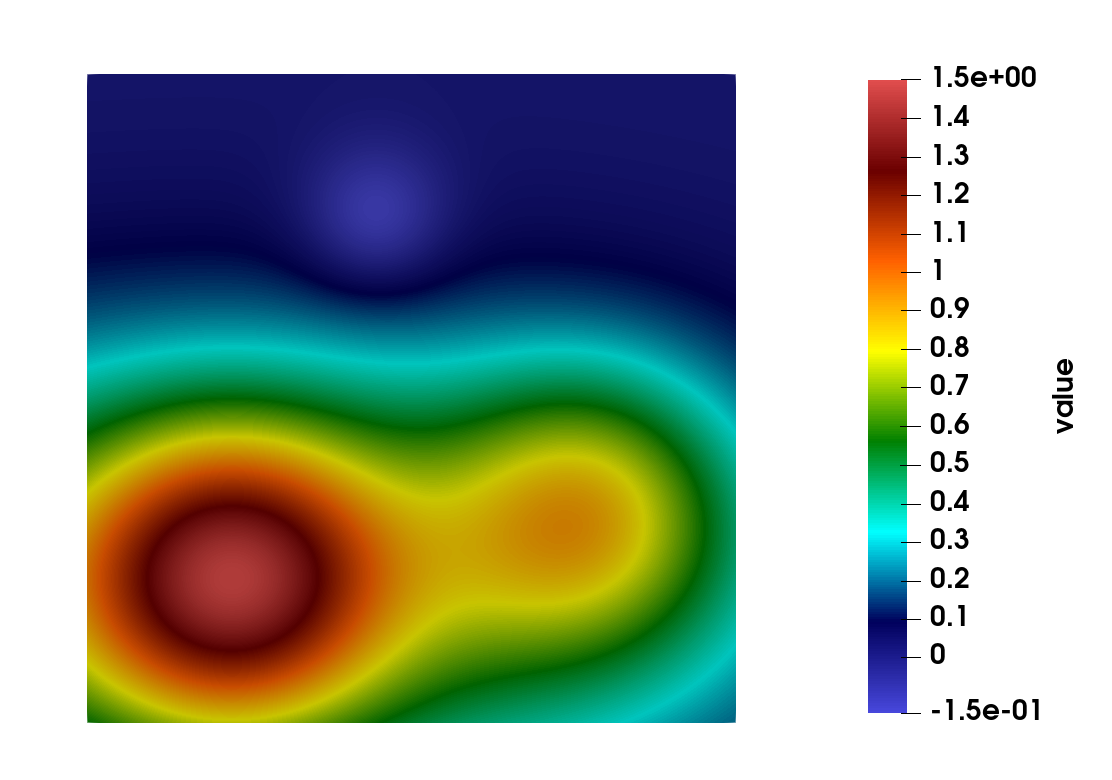}};
           \draw(6,-13.5) node {
     \includegraphics[width=0.39\textwidth]{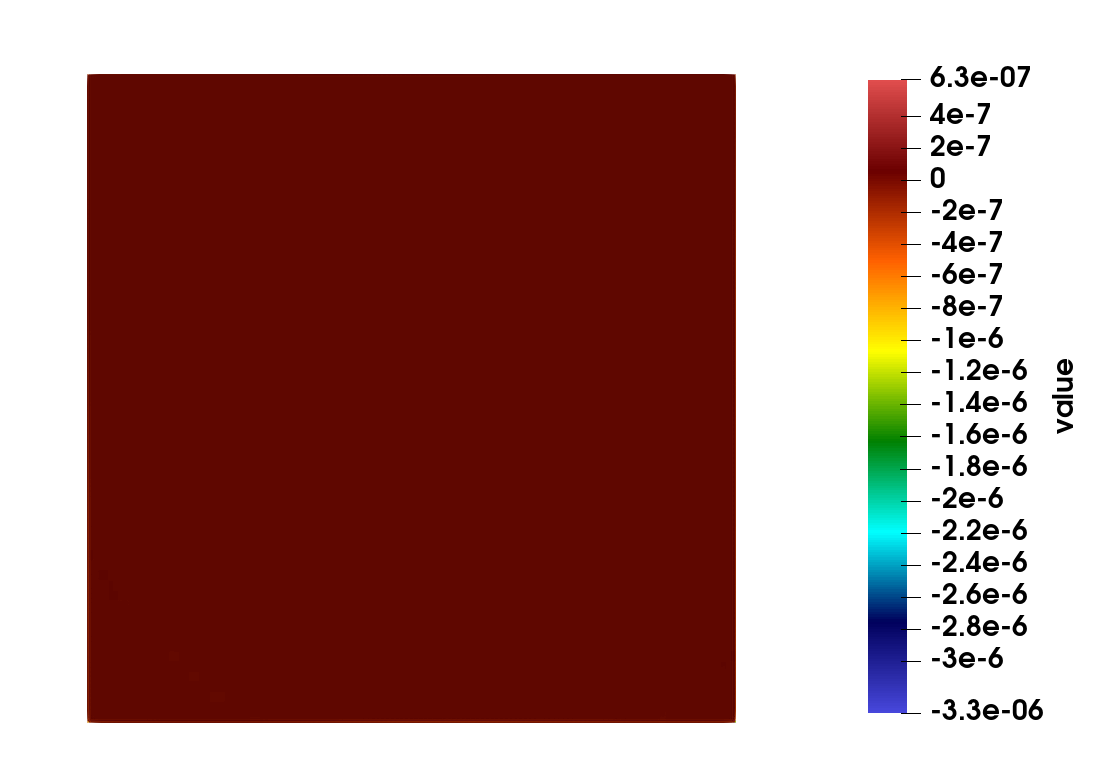}};
    \draw(3,-11.5) node {$\ell=10$};
    \end{tikzpicture}
     \caption{Solution (left) and corresponding residuals (right) for 
      the interpolation of Franke's
function for levels \(\ell=1, 4, 7, 10\).}
    \label{fig:PlotFranke}
\end{figure}

\subsection{Non-smooth function} \label{NonSmoothSection}
We consider the L-shaped domain $\Omega = \left[ -\frac{1}{2}, 
\frac{1}{2} \right]^2 \setminus \left( 0, \frac{1}{2} \right]^2$
where we introduce polar coordinates $x = r \cos \phi$, 
$y = r \sin \phi $, with $r \geq 0$ and $\phi \in [\pi/2, 2\pi]$.
In this test, we aim at interpolating the function
\(
u(r, \phi) = -r ^{\frac{2}{3}} \sin \left( \frac{2\phi - \pi}{3} \right),
\)
which solves the Laplace equation $\Delta u = 0$. We select $h_\ell = 2^{-\ell}$ as
before. Given the reduced smoothness of the function, we choose the 
${C}^0$-smooth Mat\'ern-1/2 RBF for its approximation. \textcolor{blue}{The error is evaluated on a fine grid with mesh size \( h = 2^{-9} \).
The parameters for the samplet matrix compression and
the accuracy of the conjugate gradient method are the same as in the previous example.
}
Table \ref{tab:PointsLShape} provides the number of points per level along with the corresponding fill distances. The results for $\gamma = 0.5$ and $\gamma = 1$ are summarized in Table \ref{tab:Lshape1}.

\begin{table}[htb]
    \centering
    \begin{tabular}{rrl}
        \toprule
        $\ell$ & \textbf{$N_\ell$} & \textbf{$h$} \\
        \midrule
        1 & 21 & 0.25 \\
        2 & 65 & 0.125 \\
        3 & 225 & 0.0625 \\
        4 & 833 & 0.03125 \\
        5 & 3\,201 & 0.015625 \\
        6 & 12\,545 & 0.0078125 \\
        7 & 49\,665 & 0.00390625 \\
        8 & 197\,633 & 0.00195312 \\
        \bottomrule
    \end{tabular}
    \caption{Number of points per level and the corresponding 
      fill-distances on the L-shaped domain for the interpolation of
      the solution to the Laplace equation.}
       \label{tab:PointsLShape}
\end{table}

\begin{table}[htb]
\centering
\begin{tabular}{|c|c|c|c|c|c|c|}
\hline 
\multirow[t]{2}{*}{ } & \multicolumn{6}{c|}{$\gamma=0.5$} \\
\hline 
$\ell$ &  $\text{error}_2$ & $\text{error}_\infty$ & 
$\text{order}_{2}$ & $\text{order}_{\infty}$ & \%nz & CG \\
\hline 
1 & $ 4.37 \cdot 10^{-2}$ & $ 5.98 \cdot 10^{-2}$ & -- &  --& 52 & 8 \\

2 &  $ 1.38 \cdot 10^{-2}$ & $ 3.17 \cdot 10^{-2}$ & 1.66 & 0.92 & 49 & 15\\

3 &   $ 5.33 \cdot 10^{-3}$ & $ 2.04 \cdot 10^{-2}$ & 1.37 & 0.64 & 38 & 22\\

4 &  $ 1.83 \cdot 10^{-3}$ & $ 1.29 \cdot 10^{-2}$ & 1.54 & 0.66 & 18 & 25 \\

5 &  $ 5.88 \cdot 10^{-4}$ & $ 8.15 \cdot 10^{-3}$ & 1.64 & 0.66 & 6.2 & 27 \\

6 & $ 1.85 \cdot 10^{-4}$ & $ 5.06 \cdot 10^{-3}$ & 1.67 & 0.69 & 1.8 & 28 \\

7 & $ 5.84 \cdot 10^{-5}$ & $ 3.17 \cdot 10^{-3}$ & 1.66 & 0.67 & 0.51 & 28 \\

8 & $ 1.84 \cdot 10^{-5}$ & $ 1.59 \cdot 10^{-3}$ & 1.66 & 1.00 & 0.13 & 28 \\
\hline
\end{tabular}

\vspace{0.3 cm}

\centering
\begin{tabular}{|c|c|c|c|c|c|c|}
\hline 
\multirow[t]{2}{*}{ } & \multicolumn{6}{c|}{$\gamma=1$} \\
\hline 
$\ell$ &  $\text{error}_2$ & $\text{error}_\infty$ & 
$\text{order}_{2}$ & $\text{order}_{\infty}$ & \%nz & CG \\
\hline 
1 & $ 3.39 \cdot 10^{-2}$ & $ 4.81 \cdot 10^{-2}$ & -- &  -- & 52 & 9 \\

2 & $ 1.27 \cdot 10^{-2}$ & $ 3.12 \cdot 10^{-2}$ & 1.42 & 0.62 & 51 & 21 \\

3 & $ 4.15 \cdot 10^{-3}$ & $ 1.98 \cdot 10^{-2}$ & 1.61 & 0.66 & 48 & 37 \\

4 & $ 1.29 \cdot 10^{-3}$ & $ 1.25 \cdot 10^{-2}$ & 1.69 & 0.66 & 39 & 54 \\

5 & $ 3.97 \cdot 10^{-4}$ & $ 7.90 \cdot 10^{-3}$ & 1.70 & 0.66 & 19 & 69 \\

6 & $ 1.22 \cdot 10^{-4}$ & $ 4.90 \cdot 10^{-3}$ & 1.70 & 0.69 & 6.7 & 77 \\

7 & $ 3.79 \cdot 10^{-5}$ & $ 3.04 \cdot 10^{-3}$ & 1.69 & 0.69 & 1.9 & 79 \\

8 & $ 1.13 \cdot 10^{-5}$ & $ 1.15 \cdot 10^{-3}$ & 1.75 & 1.40 & 0.44 & 79 \\
\hline
\end{tabular}
    \caption{Results for the interpolation of the solution of the Laplace
    equation on the L-shaped domain using the ${C}^0$-smooth Mat\'ern-1/2 RBF
    and $\rho = 2$, $q= 4$ as samplet compression 
  parameters.}
    \label{tab:Lshape1}
\end{table}

The compression rates are similar to the previous example, also in case of
the less smooth Mat\'ern-1/2 RBF. The approximation error and the corresponding
order are reduced, however. The number of conjugate gradient iterations
stays rather low with a maximum of 28 iterations for $\gamma=0.5$ and of 79 
iterations for $\gamma=1$. This can be explained by
the slower decay of the Mat\'ern-1/2 kernel's eigenvalues, resulting in less
ill-conditioned generalized Vandermonde matrices.

Figure \ref{fig:PlotLshape} illustrates the solution (left)
and the corresponding residuals (right) at levels \(\ell=2, 4, 6, 8\) with
$\gamma = 0.5$.

\begin{figure}[htb]
  \centering
  \begin{tikzpicture}
    \draw(0,0) node {
      \includegraphics[width=0.4\textwidth]{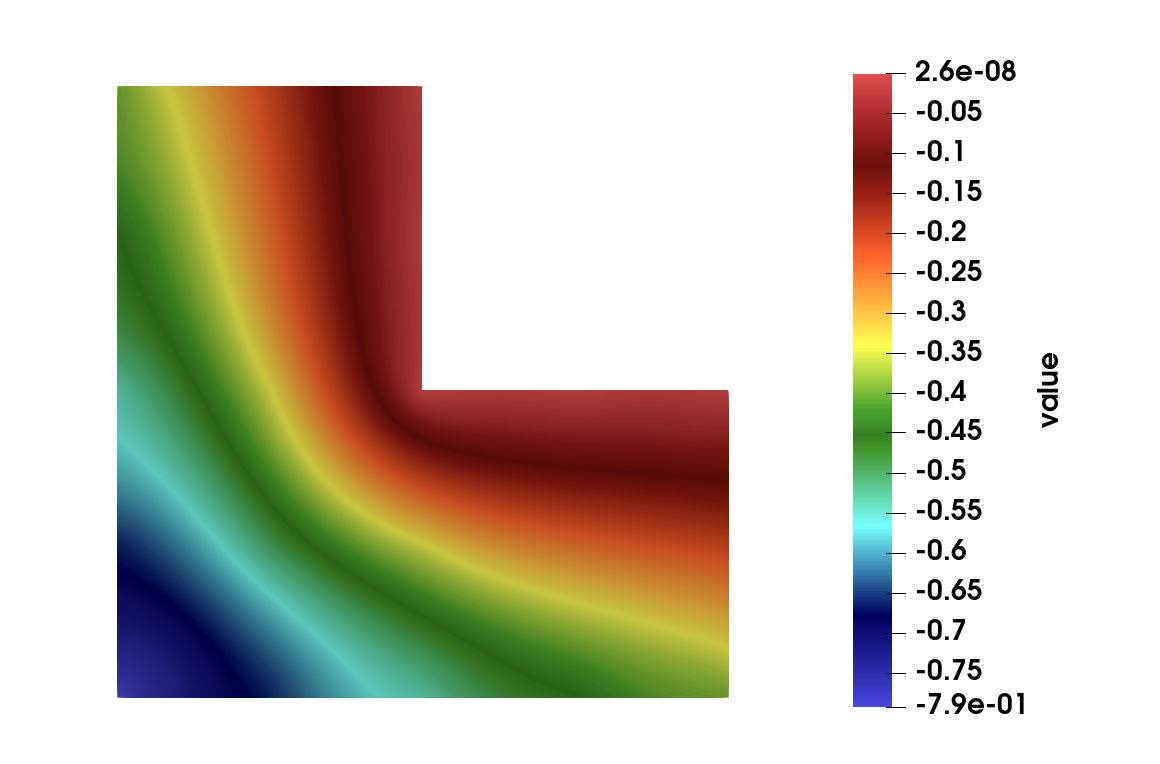}};
      \draw(6,0) node {
          \includegraphics[width=0.4\textwidth]{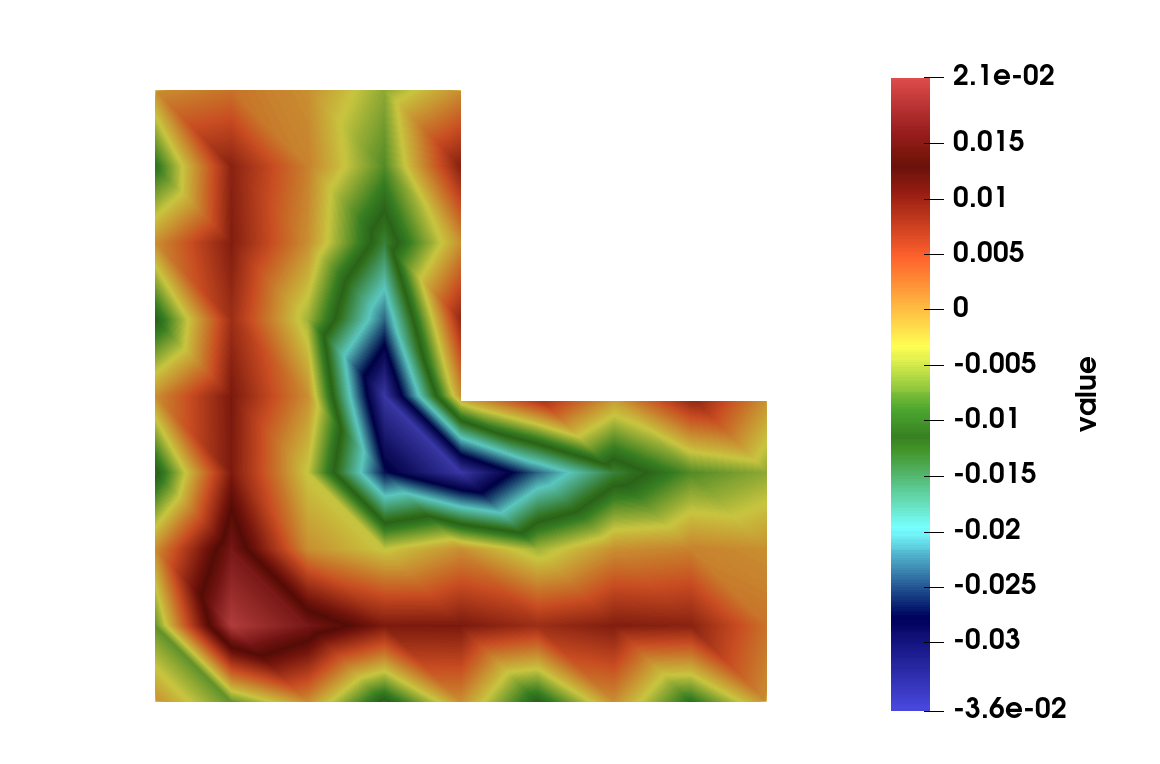}};
            \draw(3,2) node {$\ell=2$};
      \draw(0,-4.5) node {
          \includegraphics[width=0.4\textwidth]{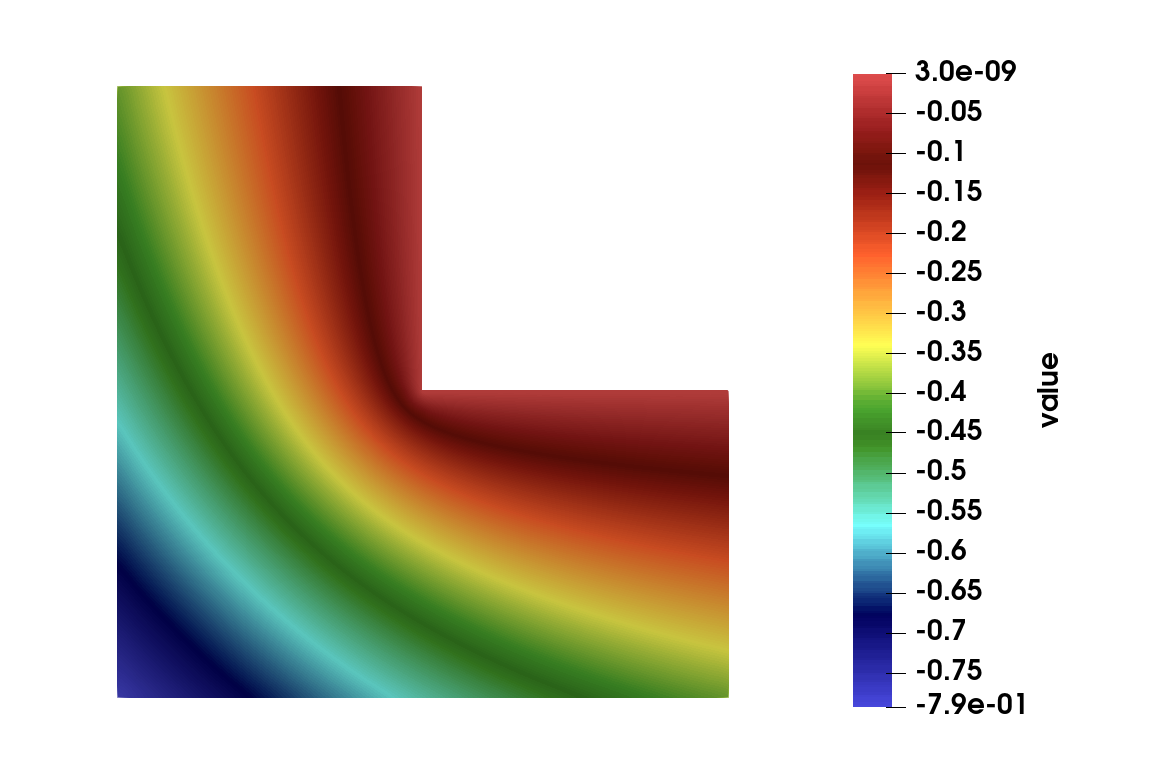}};
        \draw(6,-4.5) node {
        \includegraphics[width=0.4\textwidth]{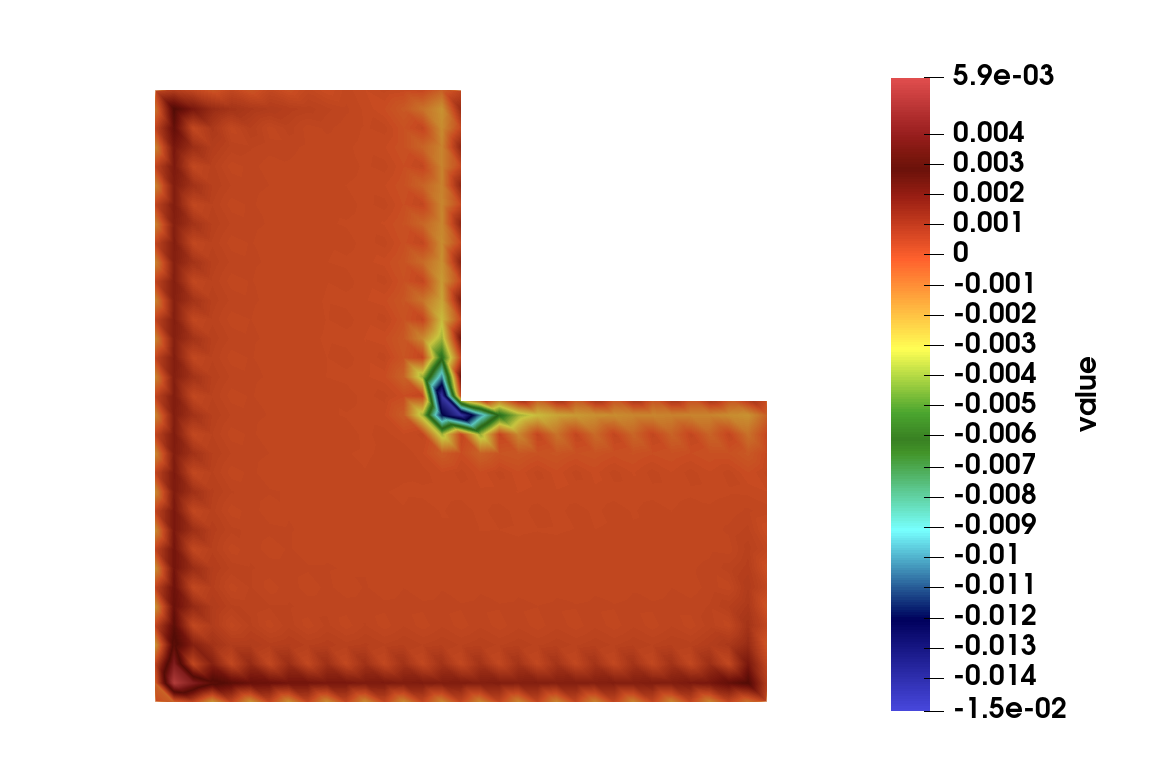}};
          \draw(3,-2.5) node {$\ell=4$};
          \draw(0,-9) node {
      \includegraphics[width=0.4\textwidth]{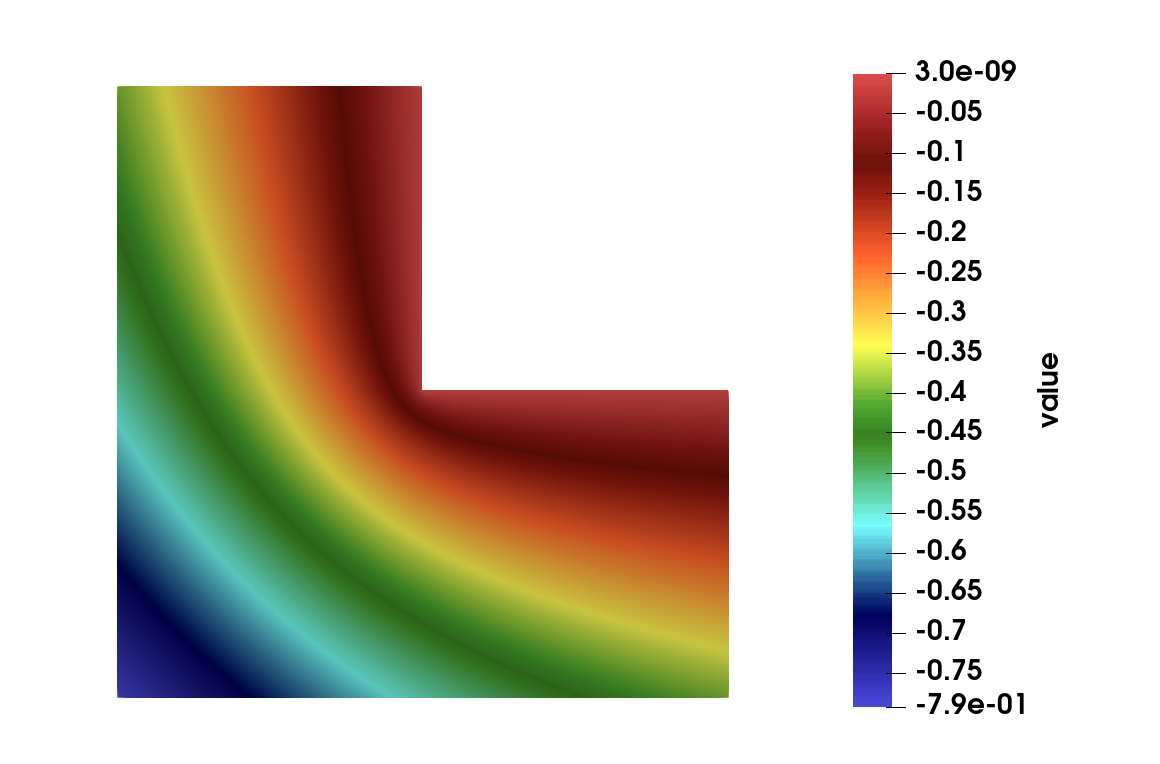}};
          \draw(6,-9) node {
      \includegraphics[width=0.4\textwidth]{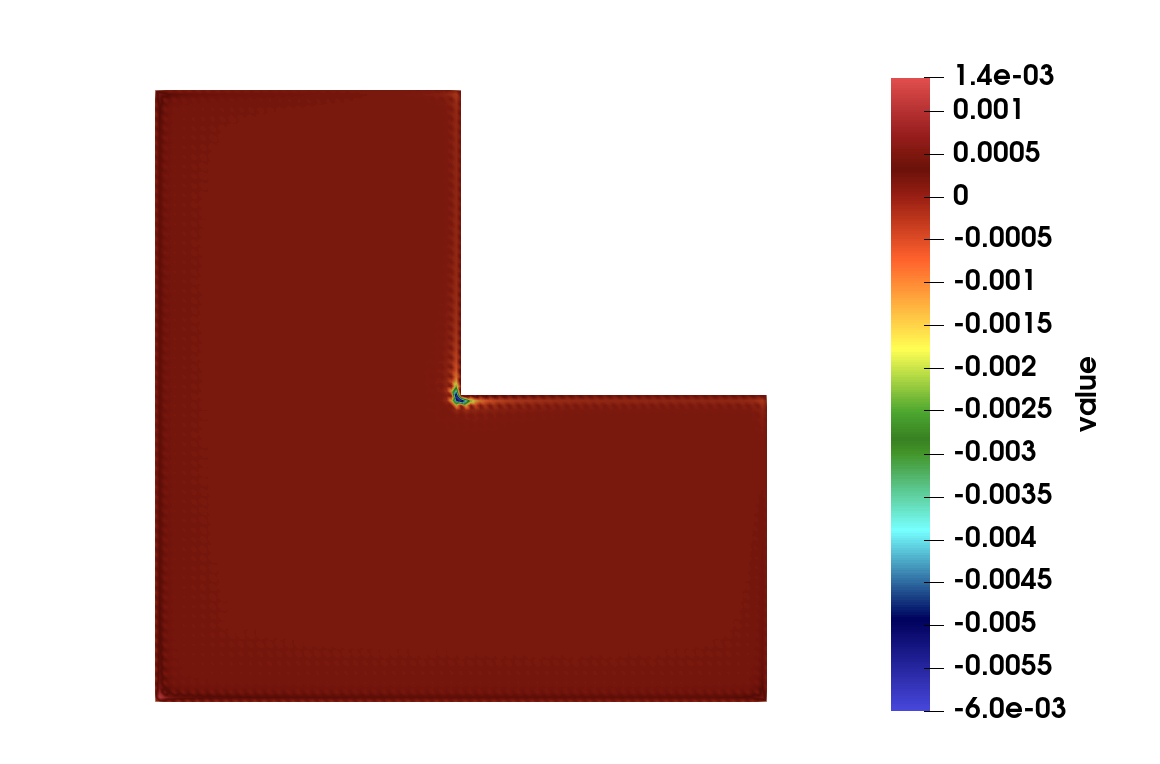}};   
    \draw(3,-7) node {$\ell=6$};
          \draw(0,-13.5) node {
      \includegraphics[width=0.4\textwidth]{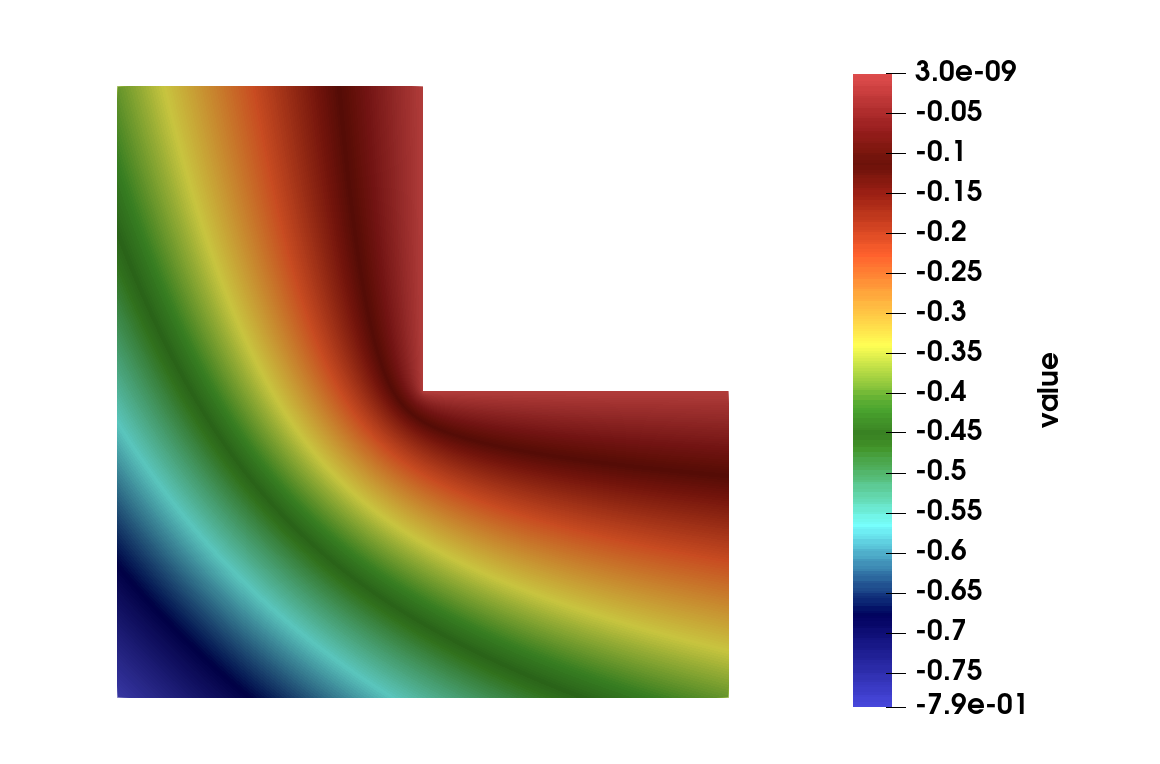}};
           \draw(6,-13.5) node {
     \includegraphics[width=0.4\textwidth]{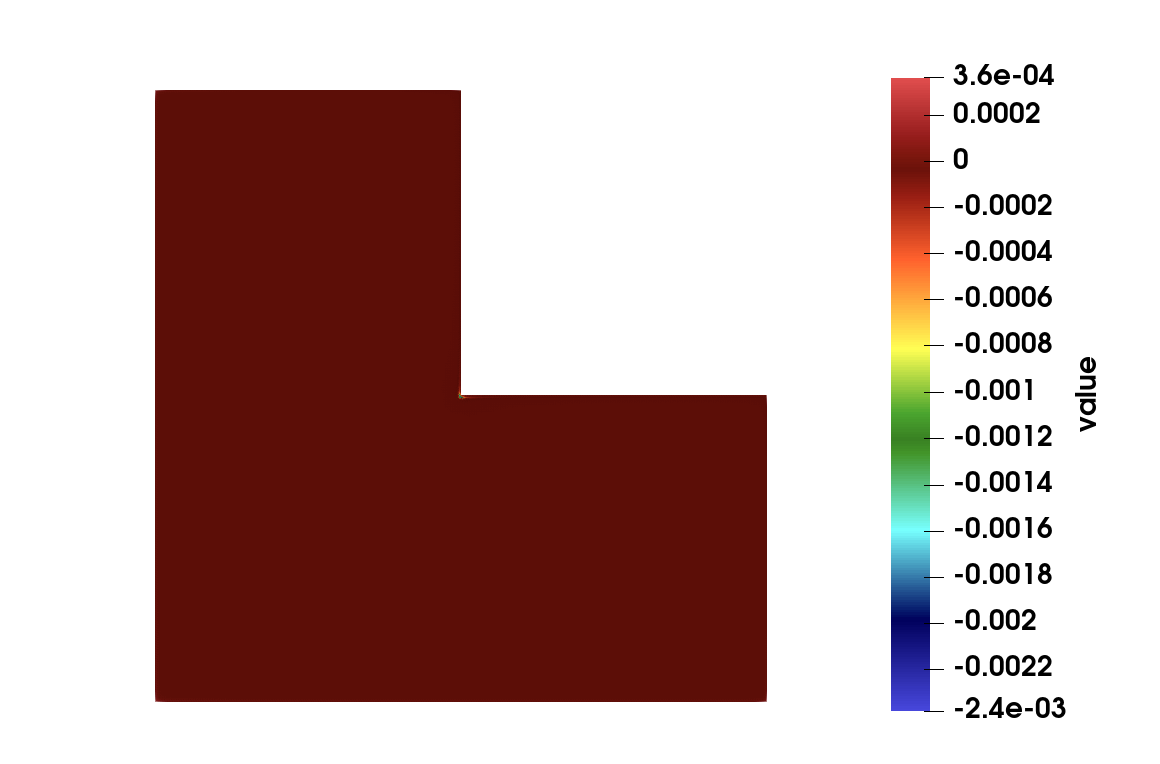}};
    \draw(3,-11.5) node {$\ell=8$};
    \end{tikzpicture}
    \caption{Solution (left) and corresponding residuals (right) for the 
      interpolation of the solution to the Laplace equation on the L-shaped
      domain for levels 
  \(\ell=2, 4, 6, 8\).}
    \label{fig:PlotLshape}
\end{figure}

\subsection{Function on a point cloud in three spatial dimensions}
\label{LucySection}
As final example, we consider a function defined on a uniform resampling
of the Lucy surface model of the \texttt{Stanford 3D Scanning Repository}. 
To construct the local approximation spaces, we selected random subsamples
of the resulting point cloud, ensuring a specific cardinality for each set. 
We start with an initial set of 20 points, and at each subsequent level, 
the cardinality is increased by a factor of 16. In this study, we present 
results for both nested and non-nested point sets. For the nested case, the 
indices of the points are randomly shuffled, and subsamples are obtained by 
selecting the first $n$ points, where $n$ represents the desired cardinality. 
In contrast, for the non-nested case, each subsample is randomly selected 
from the entire dataset 
independently, i.e., without enforcing the inclusion of points from previous 
levels. Unlike previous tests, this experiment
is more general, as the points are not arranged on a uniform grid, they are not
necessarily nested, and they are not necessarily quasi-uniform. 
As test function, we interpolate 
\(
  f({\bs x}) = \big(\sqrt[4]{\|{\bs x}-{\bs x}_0\|_2} + 10^{-4}\big)^{-1}
\)
The point ${\bs x}_0$ is selected near the torch that Lucy holds in her left
hand.
To optimize performance, we use a relatively small value for the parameter
$\gamma$ and set $\rho = 3$, $q=2$, while the a-posteriori threshold is set 
to \(10^{-5}\). These adjustments help decrease both computational time and the
number of non-zero entries. However, a very small $\gamma$ can compromise 
the precision, while a very large $\gamma$ increases the number of conjugate 
gradient iterations, thereby increasing computational time. A good 
trade-off is given by $\gamma = 0.2$. Furthermore, we employed a diagonally
preconditioned conjugate gradient method with accuracy $10^{-6}$ to solve the
linear system, which reduces the number of iterations by a factor of 
about three compared to the unpreconditioned case. The number of points per 
level and the corresponding fill-distances are provided
in Table \ref{tab:PointsLucy}. 

\begin{table}[htb]
\begin{center}
    \begin{tabular}{rrll}
        \toprule
        $\ell$ & \textbf{$N_\ell$} & \textbf{$h$} (nested) & 
        \textbf{$h$} (non-nested)  \\
        \midrule
        1 & 20 & 0.25  &  0.34\\
         2 & 320 & 0.15 & 0.13\\
         3 &  5\,120 & 0.061 & 0.064\\
         4 &  81\,920 & 0.017 & 0.018\\
         5 & 1\,310\,720 & 0.0051 & 0.0050\\
        \bottomrule
    \end{tabular}
  \end{center}  
    \caption{Number of points per level and corresponding fill-distances 
    for the interpolation on the Lucy point cloud.}
       \label{tab:PointsLucy}
\end{table}

The error is evaluated on a finer set of 2\,000\,000 points, which are randomly
selected once from the original dataset. The evaluation set remains fixed 
throughout all experiments to ensure consistency in error measurement. 
The results using the Mat\'ern-1/2 and Mat\'ern-3/2 RBFs are shown in 
Tables \ref{tab:LucyResults12} and \ref{tab:LucyResults32}, respectively.

\begin{table}[htb]
\centering
\begin{tabular}{|c|c|c|c|c|c|c|c|c|}
\hline 
\multirow[t]{2}{*}{ } & \multicolumn{8}{|c|}{$\gamma=0.2$} \\
\hline 
\multirow[t]{2}{*}{ } & \multicolumn{4}{|c|}{nested} 
                      & \multicolumn{4}{|c|}{non-nested} \\
\hline 
$\ell$ &   $\text{error}_2$ &  $\text{order}_2$ & \%nz & CG &  
$\text{error}_2$ & $\text{order}_2$ & \%nz & CG \\
\hline 
1 & $5.98 \cdot 10^{-1}$ & -- &  50 & 8  & 
$4.49 \cdot 10^{-1}$ & -- & 50 & 9 \\

2  & $1.60 \cdot 10^{-1}$ & 2.58 &  44 & 26 & 
$1.27 \cdot 10^{-1}$ & 1.31 & 42 & 28 \\

3  & $3.08 \cdot 10^{-2}$ & 1.83 &  6.5 & 49 & 
$2.25 \cdot 10^{-2}$ & 2.44 & 6.9 & 45 \\

4 & $5.22 \cdot 10^{-3}$ & 1.39 &  0.51 & 56 & 
$2.36 \cdot 10^{-3}$ & 1.78 & 0.50 & 65\\

5  & $7.94 \cdot 10^{-4}$ & 1.56 &  0.03 & 57 & 
$2.51 \cdot 10^{-4}$ & 1.74 & 0.03 & 69\\

\hline
\end{tabular}
    \caption{
Results for the interpolation of a test function on the
Lucy point cloud for the $C^0$-smooth Mat\'ern-1/2 RBF
using $\rho = 3$, $q= 2$
as samplet compression parameters.}
    \label{tab:LucyResults12}
\end{table}

\begin{table}[htb]
\centering
\begin{tabular}{|c|c|c|c|c|c|c|c|c|}
\hline 
\multirow[t]{2}{*}{ } & \multicolumn{8}{|c|}{$\gamma=0.2$} \\
\hline 
\multirow[t]{2}{*}{ } & \multicolumn{4}{|c|}{nested} 
                      & \multicolumn{4}{|c|}{non-nested} \\
\hline 
$\ell$ &   $\text{error}_2$ &  $\text{order}_2$ & \%nz & CG &  
$\text{error}_2$ & $\text{order}_2$ & \%nz & CG \\
\hline 
1 & $ 6.16 \cdot 10^{-1}$ & -- &  50 & 11  & 
$4.50 \cdot 10^{-1}$ & -- & 50 & 12 \\

2  & $ 1.48 \cdot 10^{-1}$ & 2.79 &  45 & 74 & 
$1.14 \cdot 10^{-1}$ & 1.43 & 44 & 79 \\

3  & $ 2.92 \cdot 10^{-2}$ & 1.80 &  6.5 & 214 & 
$1.73 \cdot 10^{-2}$ & 2.66 & 7.0 & 247 \\

4 & $5.41  \cdot 10^{-3}$ & 1.32 &  0.51 & 341 & 
$1.32 \cdot 10^{-3}$ & 2.03 & 0.51 & 398\\

5  & $ 8.57 \cdot 10^{-4}$ & 1.53 &  0.03 & 349 & 
$1.61 \cdot 10^{-4}$ & 1.64 & 0.03 & 430\\

\hline
\end{tabular}
    \caption{
Results for the interpolation of a test function on the
Lucy point cloud for the $C^1$-smooth Mat\'ern-3/2 RBF
using $\rho = 3$, $q= 2$
as samplet compression parameters.}
    \label{tab:LucyResults32}
\end{table}

We observe that the $C^0$-smooth Matérn-1/2 RBF performs better in this
scenario, primarily due to the improved conditioning of the associated 
generalized Vandermonde matrix. 
\textcolor{blue}{In both cases, the approximation error is higher than the compression error 
$\kappa = 10^{-5}$.}

A visualization of the solution on all five levels \(\ell=1,\ldots, 5\)
using the Mat\'ern-1/2 RBF and non-nested points is depicted in 
Figure \ref{fig:Lucy}.

\begin{figure}[htb]
    \centering
  \begin{tikzpicture}
    \draw(0,0) node {
      \includegraphics[width=0.36\textwidth]{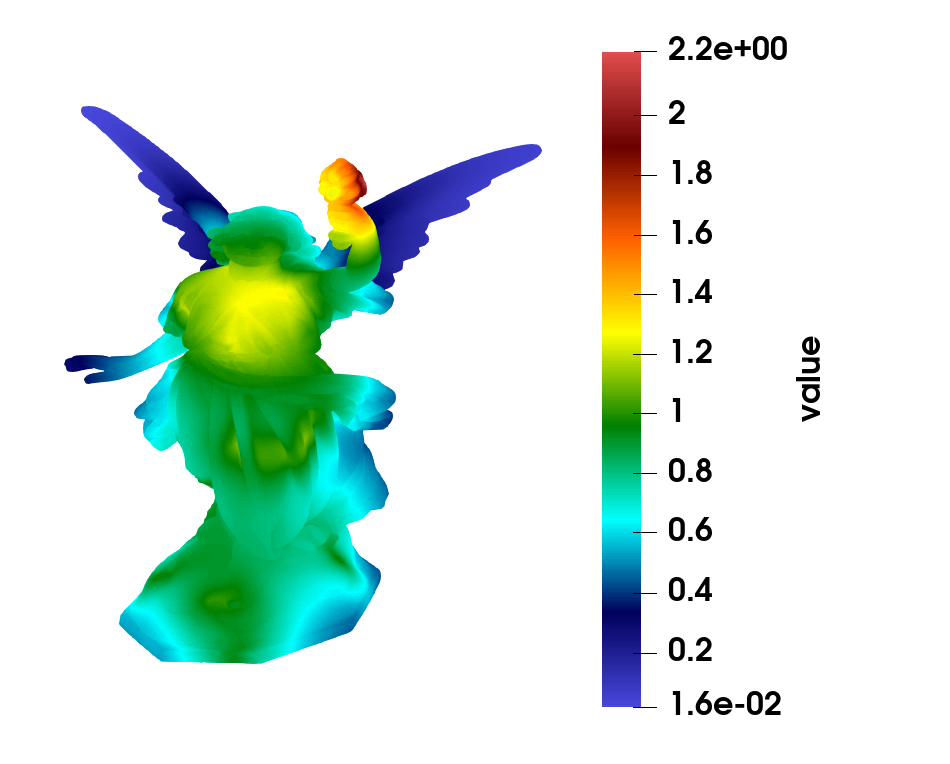}};
      \draw(6,0) node {
          \includegraphics[width=0.36\textwidth]{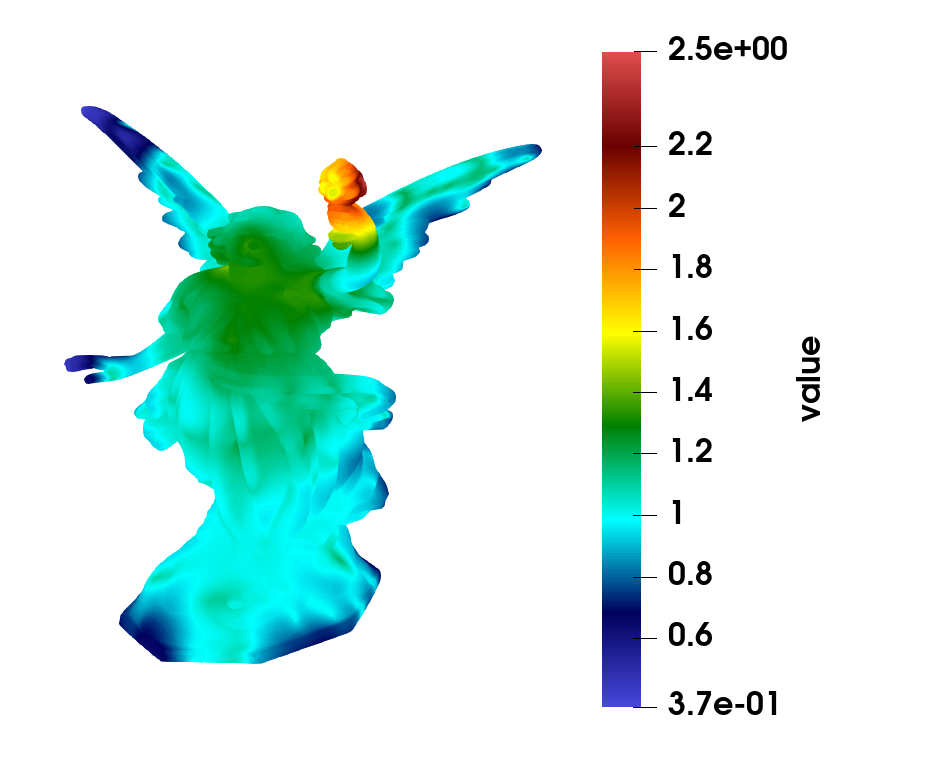}};
            \draw(0,2.25) node {$\ell=1$};
            \draw(6,2.25) node {$\ell=2$};
      \draw(0,-5) node {
          \includegraphics[width=0.36\textwidth]{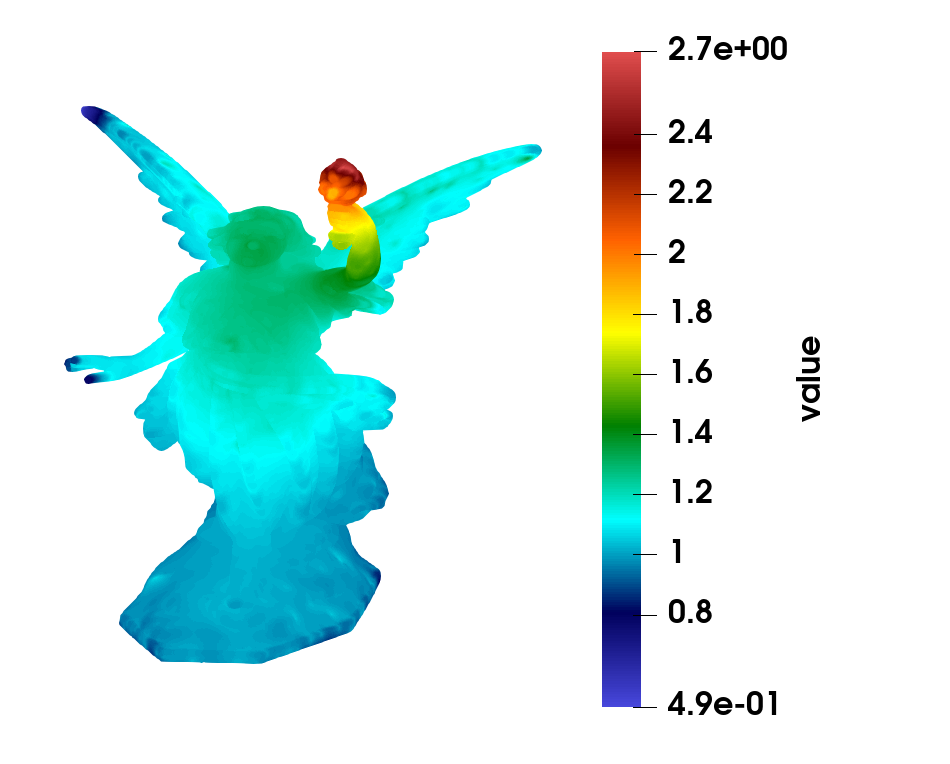}};
        \draw(6,-5) node {
        \includegraphics[width=0.36\textwidth]{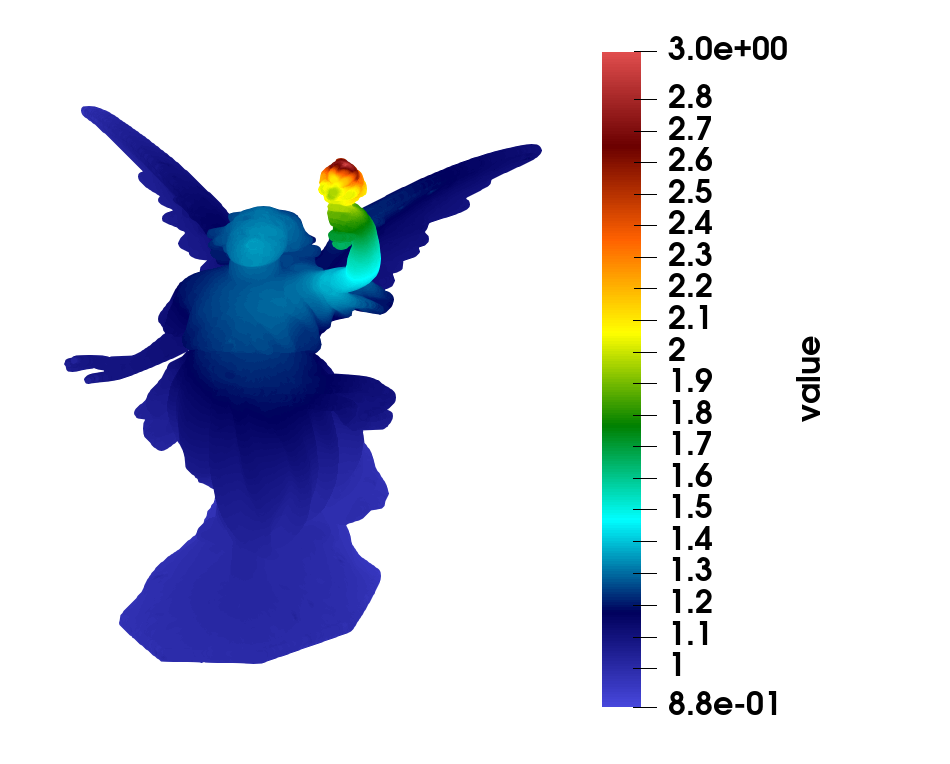}};
          \draw(0,-2.75) node {$\ell=3$};
          \draw(6,-2.75) node {$\ell=4$};
          \draw(3,-10) node {
      \includegraphics[width=0.36\textwidth]{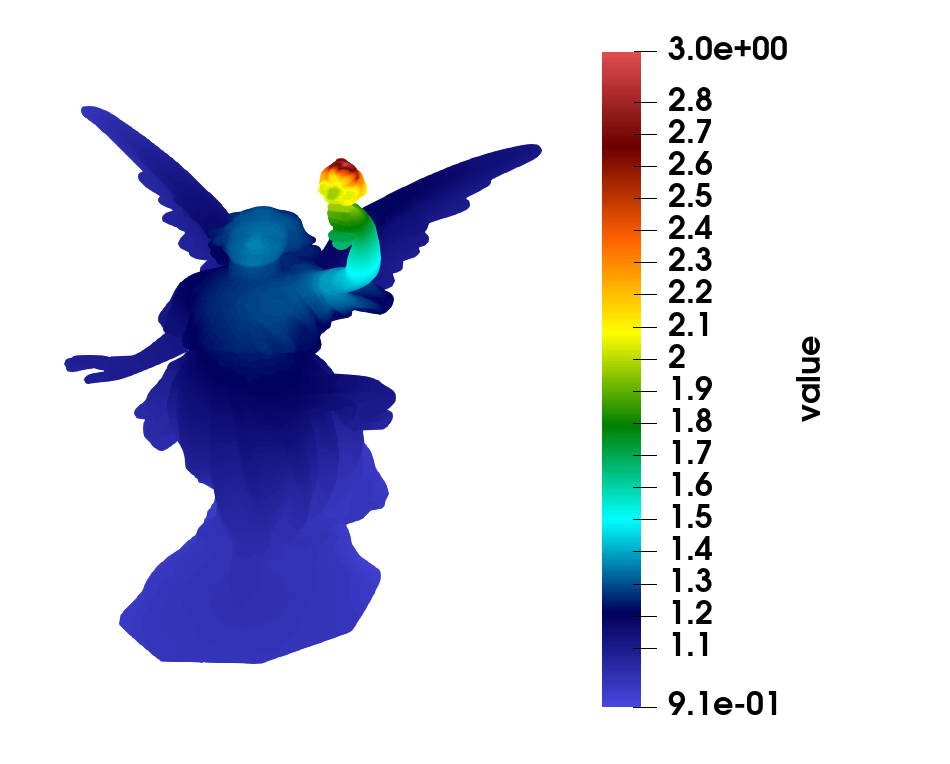}};
                \draw(3,-7.75) node {$\ell=5$};
    \end{tikzpicture}
    \caption{Interpolant of the test function on the Lucy point cloud
    for levels $\ell=1, \ldots, 5$ using the Mat\'ern-1/2 RBF.}
    \label{fig:Lucy}
\end{figure}

\section{Conclusion}\label{sec:Conclusion}
We have proposed an efficient multiscale approach for scattered data
interpolation with globally supported RBFs in combination with the samplet
matrix compression. In particular, we have extended existing results
on the boundedness of the condition numbers of the generalized Vandermonde
matrices within the multiscale approximation towards globally supported RBFs.
For appropriately chosen lengthscale parameters, the condition numbers of the
diagonal blocks of the multiscale system stay bounded
independently of the level.
\textcolor{blue}{As a consequence, the diagonally scaled
multiscale system becomes well conditioned. We have exploit this fact and
derived a general error estimate bounding the consistency error issuing from 
a numerical approximation of the multiscale system. This result particularly
allows to rigorously bound the consistency error arising from the samplet 
matrix compression. The sparse representation of the multiscale system
in samplet coordinates}
allows for an efficient solution of the
linear system by the conjugate gradient method with a bounded number
of iterations for a given accuracy, \textcolor{blue}{due to the well conditioning}. 
The resulting algorithm has an overall
cost of \(\Ocal(N\log^2 N)\) for \(N\) quasi-uniform data sites, 
given local approximation spaces with exponentially decreasing
dimension with decreasing level.
\textcolor{blue}{The numerical results demonstrate that the presented approach
achieves a similar accuracy to the multiscale approach using compactly supported RBFs, 
as presented in \cite{wendland2010multiscale}, where the corresponding experiments can 
be found. Meanwhile, our approach overcomes previous limitations that made the use of 
globally supported RBFs impractical for large-scale problems.
}
We have demonstrated that the present
multiscale approach can handle datasets with millions of data in both two and
three spatial dimensions. As test cases, we have considered smooth and
non-smooth data in two dimensions and a complex geometry in three
spatial dimensions. In all cases, we achieved accurate results that would have
been computationally infeasible without the samplet matrix compression.
